\documentclass[oneside,english]{amsart}
\usepackage[T1]{fontenc}
\usepackage[latin9]{inputenc}
\usepackage[letterpaper]{geometry}
\usepackage{color}
\usepackage{array}
\usepackage{amsthm}
\usepackage{amssymb}

\makeatletter

\providecommand{\tabularnewline}{\\}

\numberwithin{equation}{section}
\numberwithin{figure}{section}
\theoremstyle{plain}
\newtheorem{thm}{Theorem}
  \theoremstyle{plain}
  \newtheorem{prop}[thm]{Proposition}
  \theoremstyle{definition}
  \newtheorem{defn}[thm]{Definition}
  \theoremstyle{plain}
  \newtheorem{cor}[thm]{Corollary}
  \theoremstyle{plain}
  \newtheorem{lem}[thm]{Lemma}
  \theoremstyle{plain}
  \newtheorem{fact}[thm]{Fact}

\makeatother

\usepackage{babel}

\begin{document}

\title{Cohen-Macaulayness of Rees Algebras of Diagonal Ideals}

\author{Kuei-Nuan Lin }


\begin{abstract}
Given two determinantal rings over a field $k$. We consider the Rees
algebra of the diagonal ideal, the kernel of the multiplication map.
The special fiber ring of the diagonal ideal is the homogeneous coordinate
ring of the join variety. When the Rees algebra and the Symmetric
algebra coincide, we show that the Rees algebra is Cohen-Macaulay.
\end{abstract}
\maketitle

\section{Introduction}

\textcolor{black}{Determinantal rings and varieties have been a central
topic of commutative algebra and algebraic geometry. The embedded
join of two subschemes $X,$ $Y$ of $\mathcal{\mathbb{P}}_{k}^{n}$
is another important subject. When $X=Y$, the join construction yields
the classical secant variety. Join varieties are an important topic
in algebraic geometry. The embedded join of $X$ and $Y$ is the closure
of the union of all lines passing through two distinct points of $X$
and $Y$. An important question is whether the vaiety is all of $\mathcal{\mathbb{P}}_{k}^{n}$
. As the special fiber ring of the diagonal ideal is the homogeneous
coordinate ring of the embedded join, it is natural to investigate
the blowup along the diagonal, rather than just the special fiber
in the blow up.}

\textcolor{black}{To study join varieties of determinantal varieties,
we investigate blowups in products of determinantal varieties. It
turns out that for some of the cases where the embedded join is the
whole space \cite{S-U}, the Rees algebra and the symmetric algebra
of the diagonal ideal coincide \cite{L}. In this work, we show that
the Rees algebras are Cohen-Macaulay in those cases. This continues
work of Simis and Ulrich \cite{S-U}, and of the author \cite{L}.}

\textcolor{black}{We now describe the setting. Let $k$ be a field,
$2\leq m\leq n$ integers, $X=[x_{ij}]$ an $m\times n$ matrix of
variables over $k$, and $I=I_{u_{1}}(X)$, $J=I_{u_{2}}(X)$ the
ideals of $k[X]$ generated by the $u_{1}\times u_{1}$ minors of
$X$ and the $u_{2}\times u_{2}$ minors of $X$. Let $R_{1}=k[X]/I$,
$R_{2}=k[X]/J$ be two determinantal rings. We consider the diagonal
ideal $\mathbb{D}$ of $S=R_{1}\otimes_{k}R_{2}$, defined via the
exact sequence \[
0\longrightarrow\mathbb{D}\longrightarrow S\overset{_{\mathrm{mult.}}}{\longrightarrow}k[X]/(I+J)\longrightarrow0.\]
 The} ideal $\mathbb{D}$ is generated by the images of $x_{ij}\otimes1-1\otimes x_{ij}$
in $S$. The homogeneous coordinate ring of the embedded join variety
is the $k$-subalgebra of $S$ generated by the images of $x_{ij}\otimes1-1\otimes x_{ij}$.
Those elements are homogeneous of degree 1. The homogeneous coordinate
ring of the embedded joint variety $\mathcal{J}(I,J)\subseteq\mathbb{P}_{k}^{mn-1}$
of the determinantal varieties $V(I),\mbox{ }V(J)\mbox{ in }\mathbb{P}_{k}^{mn-1}$
can be identified with $\mathcal{R}(\mathbb{D})\otimes_{S}k=\mathcal{F}(\mathbb{D})$
regarding $k$ as $S/\mathfrak{m}$ where $\mathfrak{m}$ is the homogeneous
maximal ideal of $S$.

The scheme $\mbox{Proj}(\mathcal{F}(\mathbb{D}))$ is the special
fiber in the blowup $\mbox{Proj}(\mathcal{R}(\mathbb{D}))$ of $\mbox{Spec}(S)$
along $V(\mathbb{D})$. In this work, we study, more broadly, the
blowup, rather than the special fiber. 
\begin{thm}
\textcolor{black}{\label{CM} The Rees algebra $\mathcal{R}(\mathbb{D})$
is Cohen-Macaulay if $I$ and $J$ are generated by the maximal minors
of submatrices of $X$.}
\end{thm}
In \cite{L}, the defining ideals of Rees algebras of diagonal ideals
have been determined in the setting of theorem \ref{CM}. Let $\mathcal{K}$
be the defining ideal of the Rees algebra of $\mathbb{D}$. By the
proposition below, we deduce that $\mathcal{K}$ is Cohen-Macaulay
once we show $\mbox{in}(\mathcal{K})$ is Cohen-Macaulay. 
\begin{prop}
$(a)\:$\cite{E}$[15.15]$ Let $R$ be a polynomial ring over a field
$k$, $>$ be a monomial order on $R$, $I$ an ideal of $R$ and
$\mathrm{in}(I)$ the initial ideal of $I$ with respect to the term
order $>$. Let $a_{1},....,a_{r}$ be polynomials in $R$ such that
$\mathrm{in}(a_{1}),...,\mathrm{in}(a_{r})$ form a regular sequence
on $R/\mathrm{in}(I)$. Then $a_{1},...,a_{r}$ is a regular sequence
on $R/I$.

\begin{flushleft}
$\mathrm{(b)}\mathrm{\,}$\cite{E}$[15.16,15.17]$ If $R/\mathrm{in}(I)$
is Cohen-Macaulay, then so is $R/I$. 
\par\end{flushleft}
\end{prop}
We use combinatorial commutative algebra to show $\mbox{in}(\mathcal{K})$
is Cohen-Macaulay. With respect to a suitable term order, $\mbox{in}(\mathcal{K})$
is generated by square-free monomials \cite{L}. Square-free monomial
ideals in a polynomial ring are also known as Stanley\textendash{}Reisner
ideals. This leads us to consider Alexander dual ideals:
\begin{thm}
\label{ER} \textup{\cite{E-R}} Let $I$ be a square-free monomial
ideal in a polynomial ring $R$. The ring $R/I$ is Cohen-Macaulay
if and only if the Alexander dual ideal $I^{*}$ has a linear free
resolution.
\end{thm}
With Theorem \ref{ER}, we need to show that $(\mbox{in}(\mathcal{K}))^{*}$,
the Alexander dual ideal of $\mbox{in}(\mathcal{K})$, has a linear
free resolution. To do so, we find a suitable filtration starting
from the Alexander dual ideal of $\mbox{in}(\mathcal{K})$.

\medskip{}

\begin{flushleft}
\textbf{Acknowledgments:} This work is based on author's Ph. D. thesis
from Purdue University under the direction of Professor Bernd Ulrich.
The author is very grateful for so many useful suggestions from Professor
Ulrich.
\par\end{flushleft}

\section{Defining Equations of Rees Algebras}

Let $k$ be a field, $2\leq m\leq n$ integers, \textcolor{black}{$X_{mn}=[x_{ij}],\: Y_{mn}=[y_{ij}],$}
$Z_{mn}=[z_{ij}]$, $m$ by $n$ matrices of variables over $k$.
Let\textcolor{black}{{} $2\leq s_{i}\leq t_{i}$ integers, and $X_{s_{1}t_{1}}$,
$Y_{s_{2}t_{2}}$ are the submatrices of $X$ and $Y$ coming from
the first $s_{i}$ rows and first $t_{i}$ columns.} $I=I_{s_{1}}(X_{s_{1}t_{1}})$,
$J=I_{s_{2}}(X_{s_{2}t_{2}})$ the ideals of $k[X]$ generated by
the maximal minors of $X_{s_{1}t_{1}}$ and the maximal minors of
$X_{s_{2}t_{2}}$. Let $R_{1}=k[X]/I$, $R_{2}=k[X]/J$ be two determinantal
rings. We consider the diagonal ideal $\mathbb{D}$ of $R_{1}\otimes_{k}R_{2}$,
defined via the exact sequence \[
0\longrightarrow\mathbb{D}\longrightarrow S=R_{1}\otimes_{k}R_{2}\overset{_{\mathrm{mult.}}}{\longrightarrow}k[X]/(I+J)\longrightarrow0.\]
 The ideal $\mathbb{D}$ is generated by the images of $x_{ij}\otimes1-1\otimes x_{ij}$
in $R_{1}\otimes_{k}R_{2}$. 

\textcolor{black}{We write the diagonal ideal $\mathbb{D}=(\{x_{ij}-y_{ij}\})$
in \[
S=k[X_{mn},Y_{mn}]/(I_{s_{1}}(X_{s_{1}t_{1}}),I_{s_{2}}(Y_{s_{2}t_{2}}))\cong R_{1}\otimes_{k}R_{2}.\]
We have a presentation of $\mathbb{D}$,\[
\begin{array}{cccccc}
S^{l} & \overset{\phi}{\longrightarrow} & S^{mn} & \longrightarrow\mathbb{D} & \longrightarrow & 0\end{array}\]
From this we obtain a presentation of the symmetric algebra of $\mathbb{D}$,
\[
0\rightarrow(\mbox{image}(\phi))=J\longrightarrow\mbox{Sym}(S^{mn})=S[Z_{mn}]\longrightarrow\mbox{Sym}(\mathbb{D})\rightarrow0.\]
Here $J$ is the ideal generated by the entries of the row vector
$[z_{11},z_{12},...,z_{1n},....,z_{mn}]\cdot\phi$.} Hence \textcolor{black}{\[
\mbox{Sym}(\mathbb{D})\cong S[Z_{mn}]/J,\]
}where $J$ is generated by linear forms in the variables $z_{ij}$.
We write $\mathcal{R}(\mathbb{D})=S[Z_{mn}]/K$, $J\subset K$. In
general $K$ is not generated by linear forms. We can rewrite $\mbox{Sym}(\mathbb{D})=S[Z_{mn}]/J=k[X_{mn},Y_{mn,}Z_{mn}]/\mathcal{J}$
and $\mathcal{R}(\mathbb{D})=k[X_{mn},Y_{mn},Z_{mn}]/\mathcal{K}$.
\begin{thm}
\label{K}\cite{L} Notation as above. Let $X_{a_{1}...a_{s_{1}}}$
be the $s_{1}$ by $s_{1}$ submatrix of $X_{s_{1}t_{1}}$ with columns
$a_{1},...,a_{s_{1}}$, $Y_{b_{1}...b_{s_{2}}}$ the $s_{2}$ by $s_{2}$
submatrix of $Y_{s_{2}t_{2}}$ with columns $b_{1},...,b_{s_{2}}$,
$X_{a_{1}...a_{s_{1}}}^{l,k}$ the $k-l+1$ by $s_{1}$ submatrix
of $X$ with rows $l,\, l+1,..,k$ and columns $a_{1},...,a_{s_{1}}$,
and similarly for $Y$ and $Z$. 

We define \[
g_{ij,lk}=\left|\begin{array}{cc}
z_{ij} & z_{lk}\\
x_{ij}-y_{ij} & x_{lk}-y_{lk}\end{array}\right|\]
\[
f_{a_{1},...,a_{s_{1}}}=\sum_{q=1}^{s_{2}}(-1)^{q+1}\left|\left[\begin{array}{c}
Z^{q,q}\\
Y^{1,q-1}\\
X^{q+1,m}\end{array}\right]_{a_{1}...a_{s_{1}}}\right|,\]
where $1\leq a_{1}<a_{2}<...<a_{s_{1}}\leq\mathrm{min}(t_{1},t_{2})$
and $1\leq i\leq m$, $1\leq l\leq m$, $1\leq j\leq n$, $1\leq k\leq n$. 

Then $\mathcal{K}=(I_{s_{1}}(X_{s_{1}t_{1}}),\, I_{s_{2}}(Y_{s_{2}t_{2}}),\, g_{_{ij,lk}},\, f_{a_{1},...,a_{s_{1}}})$.\end{thm}
\begin{defn}
\label{f} Let $1\leq a_{1}<a_{2}<...<a_{s_{1}+k-1}\leq\mathrm{min}(t_{1},t_{2})$,
and $1\leq l\leq k\leq s_{2}$, we define $f_{a_{1},...,a_{s_{1}+k-1}}^{l,k}$
as follow: \begin{eqnarray*}
 & f_{a_{1},...,a_{s_{1}+k-1}}^{l,k}:=\\
 & {\displaystyle \sum_{r=k}^{s_{2}}}(-1)^{r+1}\left|\left[\begin{array}{c}
Z^{l,k-1}\\
Z^{r,r}\\
X^{1,l-1}\\
Y^{1,r-1}\\
Y^{r+1,s_{1}}\end{array}\right]_{a_{1},...,a_{s_{1}+k-1}}\right|\\
 & +{\displaystyle \sum_{r=k}^{s_{2}}(-1)^{r+1}\sum_{u=r+1}^{s_{1}}\left|\left[\begin{array}{c}
Z^{l,k-1}\\
X^{r,r}-Y^{r,r}\\
X^{1,l-1}\\
Y^{1,r-1}\\
Y^{r+1,u-1}\\
Z^{u,u}\\
X^{u+1,s_{1}}\end{array}\right]_{a_{1},...,a_{s_{1}+k-1}}\right|}.\end{eqnarray*}

\end{defn}
\medskip{}

\begin{defn}
\label{u}Let $1\leq p_{1}\leq m$, $1\leq q_{1}\leq n$, $ $$a_{s_{1}}<...<a_{j}\leq q_{1}<a_{j-1}<...<a_{1}$.
We define $U_{p_{1},q_{1},a_{1},...,a_{s_{1}}}$ as follows: \begin{eqnarray*}
U_{p_{1},q_{1},a_{s_{1}},...,a_{1}}:=z_{p_{1}q_{1}}\left|\left[\begin{array}{cccccc}
 &  & X^{1,p_{1}-1}\\
x_{p_{1}a_{s_{1}}} & ... & x_{p_{1}a_{j}} & y_{p_{1}a_{j-1}} & ... & y_{p_{1}a_{1}}\\
 &  & Y^{p_{1}+1,s_{1}}\end{array}\right]\right|\\
+\sum_{k=j+1}^{m}(x_{p_{1}q_{1}}-y_{p_{1}q_{1}})(-1)^{k+p_{1}}z_{p_{1}a_{k}}|X_{a_{1},...,\hat{a_{k}},..a_{m}}^{1,...,\hat{p_{1}},...,m}|+\\
\sum_{u=p_{1}+1}^{s_{1}}(x_{p_{1}q_{1}}-y_{p_{1}q_{1}})\left|\left[\begin{array}{cccccc}
 &  & X^{1,p_{1}-1}\\
x_{p_{1}a_{s_{1}}} & ... & x_{p_{1}a_{j}} & y_{p_{1}a_{j-1}} & ... & y_{p_{1}a_{1}}\\
 &  & Y^{p_{1}+1,u-1}\\
 &  & Z^{u,u}\\
 &  & X^{u+1,s_{1}}\end{array}\right]_{a_{1},...,a_{s_{1}}}\right|.\end{eqnarray*}

\end{defn}
\medskip{}

\begin{defn}
\label{w}Let $1\leq b_{s_{2}}<...<b_{1}\leq$, $1\leq p_{1}\leq m$,
$1\leq q_{1}\leq n$, $a_{s_{1}}<...<a_{s_{2}+1}<a_{p_{1}}<...<a_{1}$
and $a_{p_{1}}\leq q_{1}$. Let $i$ be integer such that $1\leq i\leq p$
and $a_{s_{2}+1}<b_{s_{2}}<...<b_{i+1}<a_{p_{1}-1}\leq b_{i}$ and
$b_{l}\neq a_{p_{1}}$ for $l\geq i+1$. 

We define $M_{12}$ as follows:\[
M_{12}=z_{p_{1}q_{1}}x_{p_{1}a_{p_{1}}}\left|\left[\begin{array}{c}
X^{1,p_{1}-1}\\
Y^{s_{2}+1,s_{1}}\end{array}\right]_{a_{1},...,a_{p_{1-1}},a_{s_{2}+1},...,a_{s_{1}}}\right|.\]

We define \begin{eqnarray*}
 & W_{p_{1},q_{1},a_{1},...,a_{p1},a_{s_{2}+1},...,a_{s_{1}},b_{1},...,b_{s_{2}}}:=\\
 & M_{12}|Y_{b_{1},...,b_{s_{2}}}^{1,s_{1}}|-\left|Y_{b_{1},..,b_{i}}^{1,i}\right|{\displaystyle \sum_{\{c_{i+1},...,c_{p_{1}},d_{p_{1}+1},...,d_{s_{2}}\}=\{b_{i+1},....,b_{s_{2}}\}}}\\
 & \left|Y_{c_{i+1},..,c_{p_{1}}}^{i+1,p_{1}}\right|U_{p_{1}q_{1},a_{1},...,a_{p_{1-1}},a_{p_{1}},d_{p_{1}+1},...,d_{s_{2}},a_{s_{2}+1},...,a_{s_{1}}}.\end{eqnarray*}

\medskip{}

\label{wp}Let $1\leq p_{1}\leq m$, $1\leq q_{1}\leq n$, $v=p_{1}+1,...,s_{2}-1$,
$1\leq a_{s_{1}}<...<a_{s_{2}+1}<a_{p_{1}}<...<a_{1}\leq t_{1}$ and
$a_{p_{1}}\leq q_{1}$. Let $i$ be an integer, $1\leq i\leq p$ and
let $a_{s_{2}+1}<b_{s_{2}}<...<b_{v+2}<b_{v}^{'}<...<b_{p_{1}+1}^{'}<b_{p_{1}}<...<b_{i+1}<a_{p_{1}-1}\leq b_{p_{1}+1}$
and $b_{l}^{'}\neq a_{p_{1}}$ for $l\geq i+1$ and $b_{v-1}^{'}\leq b_{v+1}$.
Let $a_{s_{1}}<....<a_{s_{2}+1}<b_{s_{2}}<...<b_{v+2}<b_{v+1}<b_{v}<b_{v-1}<...<b_{p_{1}+2}<a_{p_{1}}<a_{p_{1}-1}\leq b_{p_{1}+1}$,
and $b_{r}^{'}\leq b_{r+2}<b_{r+1}$ for $r=p_{1},...,v-2$. 

We define \begin{eqnarray*}
 & W_{p_{1},q_{1},a_{1},...,a_{p1},a_{s_{2}+1},...,a_{s_{1}},b_{1},...,b_{p_{1}+1},b_{p_{1}+2},b_{p_{1}+3}.,b_{s_{2}},b_{p_{1}+1}^{'},b_{p_{1}+2}^{'},...,b_{v}^{'}}^{p_{1}+1,v}:=\\
 & y_{v-1,b_{v-1}^{'}}y_{v,b_{v}}W_{p_{1},q_{1},a_{1},...,a_{p1},a_{s_{2}+1},...,a_{s_{1}},b_{1},...,b_{v-1},b_{v}^{'},b_{v+1},...,b_{s_{2}},b_{p_{1}+1}^{'},b_{p_{1}+2}^{'},...,b_{v-2}^{'}}^{p_{1}+1,v-2}\\
 & -y_{v,b_{v}^{'}}W_{p_{1},q_{1},a_{1},...,a_{p1},a_{s_{2}+1},...,a_{s_{1}},b_{1},...,b_{s_{2}},b_{p_{1}+1}^{'},b_{p_{1}+2}^{'},...,b_{v-1}^{'}}^{p_{1}+1,v-1},\end{eqnarray*}
where \[
W_{p_{1},q_{1},a_{1},...,a_{p_{1}},a_{s_{2}+1},...,a_{s_{1}},b_{1},...,b_{s_{2}}}^{p_{1}+1,p_{1}-1}=U_{p_{1},q_{1},a_{1},...,a_{s_{1}}}\]
 and \[
W_{p_{1}q_{1},a_{1},...,a_{p_{1}},a_{s_{2}+1},...,a_{s_{1}},b_{1},...,b_{s_{2}}}^{p_{1}+1,p_{1}}=W_{p_{1}q_{1},a_{1},..,a_{p_{1}},a_{s_{2}+1},...,a_{s_{1}},b_{1},...,b_{s_{2}}}.\]

\end{defn}
\medskip{}

\begin{defn}
\label{v}Let $b_{s_{1}}<...<b_{1}$, and $1\leq p_{l}<...<p_{k}<b_{s_{1}}<...<b_{s_{2}+1}<c_{s_{2}}<...<c_{k+1}<b_{k-1}<...<b_{1}<a_{k-1}<...<a_{1}\leq t_{1}$. 

Let \[
M_{12}=\left|\left[\begin{array}{c}
Z^{l,k}\\
X^{1,k-1}\\
Y^{s_{2}+1,s_{1}}\end{array}\right]_{p_{l},..,p_{k},a_{1},...,a_{k-1},b_{s_{2}+1},...,b_{s_{1}}}\right|.\]

We define \begin{eqnarray*}
 & V_{p_{l},...,p_{k},a_{1},...,a_{k-1},b_{1},...,b_{s_{1}}}:=\\
 & M_{12}|Y_{b_{1},...,b_{s_{2}}}^{1,s_{2}}|{\displaystyle -\sum_{\{e_{k},c_{k+1},...,c_{s_{2}}\}=\{b_{k},....,b_{s_{2}}\}}}\\
 & \pm y_{ke_{k}}f_{p_{l},...,p_{k},a_{1},...,a_{k-1},b_{1},...,b_{k-1},c_{k+1},...,c_{s_{2}},b_{s_{2}+1},...,b_{s_{1}}}^{l,k}.\end{eqnarray*}

\end{defn}
\medskip{}

\begin{defn}
\label{vk}Let $1\leq l\leq k\leq s_{2}$ and $1\leq p_{l}<...<p_{k}<b_{s_{1}}<...<b_{s_{2}+1}<...<b_{k+1}<b_{k-1}<b_{k}<b_{k-2}<...<b_{1}<a_{l-1}<...<a_{1}\leq t_{1}$.
Let $w=k,...,s_{2}-1$ and $1\leq b_{s_{2}}<...<b_{w+2}<b_{w}^{'}<b_{w-1}^{'}<...<b_{k}^{'}<b_{k-1}<b_{k-2}...<b_{1}\leq t_{2}$
and $b_{w-1}^{'}\leq b_{w+1}$, and $b_{r}^{'}\leq b_{r+2}<b_{r+1}$
for $r=k,...,l-2$. 

We define \begin{eqnarray*}
 & V_{p_{l},...,p_{k},a_{1},...,a_{k-1},b_{1},...,b_{s_{1}},b_{k}^{'},b_{k+1}^{'},...,b_{w}^{'}}^{k,w}:=\\
 & y_{w-1,b_{w-1}}y_{w,b_{w}}V_{p_{l},...,p_{k},a_{1},...,a_{k-1},b_{1},..,b_{w}^{'},...,b_{s_{1}},b_{k}^{'},b_{k+1}^{'},...,b_{w-2}^{'}}^{k,w-2}\\
 & -y_{wb_{w}^{'}}V_{p_{l},...,p_{k},a_{1},...,a_{k-1},b_{1},...,b_{s_{1}},b_{k}^{'},b_{k+1}^{'},...,b_{w-1}^{'}}^{k,w-1} & .\end{eqnarray*}
Here $ $$V_{p_{l},...,p_{k},a_{1},...,a_{k-1},b_{1},...,b_{s_{1}}}^{k,k-2}=V_{p_{l},...,p_{k},a_{1},...,a_{k-1},b_{1},...,b_{s_{1}}}^{k,k-1}=V_{p_{l},...,p_{k},a_{1},...,a_{k-1},b_{1},...,b_{s_{1}}}$.
\end{defn}
\medskip{}

\begin{defn}
\label{h}Let $1\leq l\leq k\leq s_{2}$, $1\leq q\leq n$, $1\leq a_{s_{1}+k-1}<....<a_{1}\leq t_{1}$,
$a_{s_{1}+k-1}<q$, $a_{j+1}\leq q<a_{j}$ for some $j=l-1,...,s_{1}+k-3$.
Let $\overline{f_{a_{1},...,\hat{a_{c}},...,a_{s_{1+k-1}}}^{l,k,x_{l-1}}}$
be the determinant of matrices that coming from deleting row $x_{l-1}$
and column $a_{c}$. We define $H_{a_{1},....,a_{s_{1}+k-1}}^{l,k,q}$
as following \begin{eqnarray*}
H_{a_{1},...,a_{s_{1}+k-1}}^{l,k,q} & = & z_{l-1,q}f_{a_{1},...,a_{s_{1}+k-1}}^{l,k}-\sum_{c=k}^{j}(-1)^{k+c}g_{l-1,q,l-1,a_{c}}\overline{f_{a_{1},...,\hat{a_{c}},...,a_{s_{1+k-1}}}^{l,k,x_{l-1}}}.\end{eqnarray*}

\end{defn}
\medskip{}

\begin{defn}
\label{i}Let $1\leq l\leq k\leq s_{2}$, $1\leq q\leq n$, $1\leq a_{s_{1}+k-1}<....<a_{1}\leq t_{1}$,
$a_{s_{1}+k-1}<q$, $a_{j+1}\leq q<a_{j}$ for some $j=l-1,...,s_{1}+k-3$.
Let $a_{l+s_{2}-1}<b_{s_{2}}<...<b_{k}<a_{l-1+k-1}=b_{k-1}<....<a_{l-1+1}=b_{1}$. 

Let \[
M=z_{l-1,q}x_{l-1,a_{j+1}}\left|\left[\begin{array}{c}
Z^{l,k}\\
X^{1,l-2}\\
Y^{s_{2}+1,s_{1}}\end{array}\right]_{a_{s_{1}+k-1},...,a_{l+s_{2}-1},a_{l-1},...,a_{1}}\right|.\]

We define \begin{eqnarray*}
 &  & I_{a_{s_{1}+k-1},...,a_{l+s_{2}-1},a_{l-1},...,a_{1},b_{1},...,b_{s_{2}}}^{l,k,q}:=\\
 &  & M|Y_{b_{1},...,b_{s_{2}}}^{1,s_{2}}|-\sum_{\{e_{k},c_{k+1},...,c_{s_{2}}\}=\{b_{k},....,b_{s_{2}}\}}\\
 &  & \pm y_{ke_{k}}H_{a_{s_{1}+k-1},...,a_{l+s_{2}-1},c_{s_{2}},...,c_{k+1},b_{k-1},...,b_{1},a_{l-1},...,a_{1}}^{l,k,q}.\end{eqnarray*}

\end{defn}
\medskip{}

\begin{defn}
\label{ik}Let $1\leq l\leq k\leq s_{2}$, $1\leq q\leq n$, $k\leq w\leq s_{2}-1$,
$1\leq q_{l}<...<q_{k}<b_{s_{1}}<...<b_{s_{2}}<...<b_{k+1}<b_{k-1}<b_{k}<b_{k-2}<...<b_{1}<a_{l-2}<...<a_{1}\leq t_{1}$,
$q_{l}<q$, $q_{l}<a_{l-1}\leq q$ . Let $w=k,...,s_{2}-1$ and $1\leq b_{s_{2}}<...<b_{w+2}<b_{w}^{'}<b_{w-1}^{'}<...<b_{k}^{'}<b_{k-1}<b_{k-2}...<b_{1}\leq t_{2}$
and $b_{w-1}^{'}\leq b_{w+1}$, and $b_{r}^{'}\leq b_{r+2}<b_{r+1}$
for $r=k,...,l-2$. 

We define \begin{eqnarray*}
 & ^{k,w}I_{q_{l},...,q_{k},b_{s_{1}},...,b_{s_{2}},...,b_{1},a_{l-1},...,a_{1},b_{k}^{'},...,b_{w}^{'}}^{l,k,q}:=\\
 & y_{w-1,b_{w-1}^{'}}y_{wb_{w}}{}^{k,w-2}I_{q_{l},...,q_{k},b_{s_{1}},...,b_{s_{2}},...,b_{w}^{'},...,b_{1},a_{l-1},...,a_{1},b_{k}^{'},...,b_{w-2}^{'}}^{l,k,q}\\
 & -y_{wb_{w}^{'}}{}^{k,w-1}I_{q_{l},...,q_{k},b_{s_{1}},...,b_{s_{2}},...,b_{1},a_{l-1},...,a_{1},b_{k}^{'},...,b_{w-1}^{'}}^{l,k,q},\end{eqnarray*}
where \begin{eqnarray*}
^{k,k-2}I_{q_{l},...,q_{k},b_{s_{1}},...,b_{s_{2}},...,b_{1},a_{l-1},...,a_{1}}^{l,k,q} & ={}^{k,k-1}I_{q_{l},...,q_{k},b_{s_{1}},...,b_{s_{2}},...,b_{1},a_{l-1},...,a_{1}}^{l,k,q}\\
=I_{q_{l},...,q_{k},b_{s_{1}},...,b_{s_{2}},...,b_{1},a_{l-1},...,a_{1}}^{l,k,q} &  & .\end{eqnarray*}

\end{defn}
\medskip{}

\begin{thm}
\cite{L}\label{GB}Use the notation of Definition \ref{f}, \ref{u},
\ref{w}, \ref{wp}, \ref{v}, \ref{vk}, \ref{h}, \ref{i}, \ref{ik}
and let $\mathcal{G}:=\{|X_{a_{1},..,a_{s_{1}}}^{1,s_{1}}|$, $|Y_{b_{1},...,b_{s_{2}}}^{1,s_{2}}|$,
$g_{p_{1}q_{1},p_{2}q_{2}}$, $f_{a_{1},...,a_{s_{1}+k-1}}^{l,k}$,
$U_{p_{1},q_{1},a_{1},...,a_{s_{1}}}$ , \\
$W_{p,q,a_{1},...,a_{p},a_{s_{2}+1},...,a_{s_{1}},b_{1},...,b_{s_{2}}}$,
\\
$W_{p_{1},q_{1},a_{1},...,a_{p1},a_{s_{2}+1},...,a_{s_{1}},b_{1},...,b_{p_{1}+1},b_{p_{1}+2},b_{p_{1}+3}.,b_{s_{2}},b_{p_{1}+1}^{'},b_{p_{1}+2}^{'},...,b_{v}^{'}}^{p_{1}+1,v}$,
\\
$V_{p_{l},...,p_{k},a_{1},...,a_{k-1},b_{1},...,b_{s_{1}}}$, $V_{p_{l},...,p_{k},a_{1},...,a_{k-1},b_{1},...,b_{s_{1}},b_{k}^{'},b_{k+1}^{'},...,b_{w}^{'}}^{k,w}$,
$H_{a_{1},...,a_{s_{1}+k-1}}^{l,k,q}$ ,\\
 \textup{$I_{a_{s_{1}+k-1},...,a_{l+s_{2}-1},a_{l-1},...,a_{1},b_{1},...,b_{s_{2}}}^{l,k,q}$,}
$^{k,w}I_{q_{l},...,q_{k},b_{s_{1}},...,b_{s_{2}},...,b_{1},a_{l-1},...,a_{1},b_{k}^{'},...,b_{w}^{'}}^{l,k,q}\}$. 

The $\mathcal{G}$ is a Groebner basis of $\mathcal{K}$ with respect
to the lexicographic term order and the variables ordered by $z_{ij}>x_{lk}>y_{pq}$
for any $i,j,l,k,p,q$ and $x_{ij}<x_{lk}$, $y_{ij}<y_{lk}$ if $i>l$
or $i=l$ and $j<k$ and $z_{ij}<z_{lk}$ if $i>l$ or if $i=l$ and
$j>k$.\end{thm}
\begin{cor}
\begin{flushleft}
\label{ininital}The initial ideal of \begin{eqnarray*}
\mathcal{K} & = & (h_{X}+h_{Y}+h_{g}+h_{f}+h_{U}+h_{W}+h_{W^{p,q,l}}+h_{V}+\\
 &  & h_{V^{l,k,w}}+h_{H^{l,k,q}}+h_{I^{l,k,q}}+h_{^{k,w}I^{l,k,q}}),\end{eqnarray*}
 where\\
\begin{tabular}{>{\centering}p{0.22in}>{\raggedright}p{4.5in}}
(1) & $h_{X}=(\{x_{1a_{1}}x_{2a_{2}}....x_{s_{1}a_{s_{1}}}|\:1\leq a_{s_{1}}<a_{s_{1}-1}<...<a_{1}\leq t_{1}\})$,\tabularnewline
(2) & $h_{Y}=(\{y_{1a_{1}}y_{2a_{2}}...y_{s_{2}a_{s_{2}}}|\:1\leq a_{s_{2}}<a_{s_{2}-1}<...<a_{1}\leq t_{2}\})$,\tabularnewline
(3) & $h_{g}=(\{z_{ij}x_{lk}|\: i<l\mbox{ or }i=l\mbox{ and }j<k,\})$\tabularnewline
\end{tabular}\\
\begin{tabular}{>{\centering}p{0.22in}>{\centering}p{4.5in}}
(4) & \raggedright{}$h_{f}=(\{z_{lq_{l}}z_{l+1q_{l+1}}...z_{kq_{k}}x_{1a_{1}}...x_{l-1a_{l-1}}y_{1b_{1}}...y_{k-1b_{k-1}}y_{k+1b_{k+1}}...y_{s_{1}b_{s_{1}}}|\:1\leq q_{l}<q_{l+1}<...<q_{k}<b_{s_{1}}<...<b_{k+1}<b_{k-1}<...<b_{1}<a_{l-1}<...<a_{1}\leq t_{1},\mbox{ }1\leq l<k\leq s_{2}\}),$\tabularnewline
\end{tabular} \begin{tabular}{>{\centering}p{0.22in}>{\raggedright}p{4.5in}}
(5) & $h_{U}=(\{z_{pq}x_{1a_{1}}...x_{pa_{p}}y_{p+1a_{p+1}}...y_{s_{1}a_{s_{1}}}|\: p=1,...,s_{1}-1,\:1\leq q\leq n,\:1\leq a_{s_{1}}<a_{s_{1}-1}<...<a_{1}\leq t_{1},\: a_{p}\leq q\}),$\tabularnewline
\end{tabular} \begin{tabular}{>{\centering}p{0.22in}>{\raggedright}p{4.5in}}
(6) & \raggedright{}$h_{W}=(\{z_{pq}x_{1a_{1}}...x_{pa_{p}}y_{1b_{1}}...y_{s_{1}b_{s_{1}}}|\: p=1,...,s_{1}-1,\:1\leq i\leq p,\:1\leq b_{s_{1}}<...<b_{s_{2}+1}<a_{p}<...<a_{1}\leq t_{1},\:1\leq b_{s_{1}}<b_{s_{1}-1}<...<b_{p+2}<b_{p}<b_{p-1}<...<b_{i+1}<b_{p+1}<b_{i}<...<b_{1}\leq t_{2},\:1\leq b_{s_{1}}<...<b_{p+2}<b_{p}<...<b_{i+1}<a_{p-1}<...<a_{1}\leq t_{1},\: b_{l}\neq a_{p},\: l=i+1,...,s_{1},\: a_{p-1}\leq b_{p+1},\: a_{p}\leq q\})$, \tabularnewline
(7) & \raggedright{}$h_{W^{p,q,l}}=(\{z_{pq}x_{1a_{1}}...x_{pa_{p}}y_{1b_{1}}...y_{p-1b_{p-1}}y_{pb_{p}}\mbox{ }y_{p+1b_{p+1}}y_{p+1b_{p+1}^{'}}...y_{vb_{v}}y_{vb_{v}^{'}}$\tabularnewline
 & \raggedright{}$\mbox{ }y_{v+1b_{v+1}}...y_{s_{1}b_{s_{1}}}|1\leq p\leq m,\:1\leq q\leq n,\, v=p+1,...,s_{2}-1,\,1\leq b_{s_{1}}<...<b_{s_{2}+1}<a_{p}<...<a_{1}\leq t_{1},\: a_{p}\leq q,\: b_{s_{2}+1}<b_{s_{2}}<...<b_{v+2}<b_{v}^{'}<...<b_{p+1}^{'}<b_{p-1}<...<b_{i+1}<a_{p-1}\leq b_{p+1}<b_{i}<...<b_{1}\leq t_{2},\:1\leq i\leq p,\: b_{l}^{'}\neq a_{p},\: l\geq i+1,\: b_{v-1}^{'}\leq b_{v+1},\: b_{s_{1}}<....<b_{s_{2}+1}<b_{s_{2}}<...<b_{v+2}<b_{v+1}<b_{v}<b_{v-1}<...<b_{p+2}<b_{p+1}<a_{p}<a_{p-1}\leq b_{p},\: b_{r}^{'}\leq b_{r+2}<b_{r+1},\: r=p,...,v-2\})$,\tabularnewline
\end{tabular} 
\par\end{flushleft}
\end{cor}
\begin{flushleft}
 \begin{tabular}{>{\centering}p{0.22in}>{\raggedright}p{4.5in}}
(8) & \raggedright{}$h_{V}=(\{z_{lq_{l}}z_{l+1q_{l+1}}...z_{kq_{k}}x_{1a_{1}}...x_{l-1a_{l-1}}y_{1b_{1}}...y_{k-1b_{k-1}}y_{kb_{k}}$\tabularnewline
 & \raggedright{}$y_{k+1b_{k+1}}...y_{s_{1}b_{s_{1}}}|\:1\leq l<k\leq s_{2},\:1\leq q_{l}<q_{l+1}<...<q_{k}<b_{s_{1}}<...<b_{k+2}<b_{k}<b_{k-1}<...<b_{1}<a_{l-1}<...<a_{1}\leq t_{1},\mbox{ }b_{k-1}\leq b_{k+1}\leq t_{2}-k+1\}$,\tabularnewline
\end{tabular} \begin{tabular}{>{\centering}p{0.22in}>{\raggedright}p{4.5in}}
(9) & $h_{V^{l,k,w}}=(\{z_{lq_{l}}...z_{kq_{k}}x_{1a_{1}}...x_{l-1a_{l-1}}y_{1b_{1}}...y_{k+1b_{k+1}}y_{k+1b_{k+1}^{'}}...\mbox{ }y_{wb_{w}}y_{wb_{w}^{'}}$\tabularnewline
 & $y_{w+1b_{w+1}}...y_{s_{1}b_{s_{1}}}|\mbox{ }\:1\leq l\leq k\leq s_{2},\:1\leq p_{l}<...<p_{k}<b_{s_{1}}<...<b_{s_{2}+1}<...<b_{k+1}<b_{k-1}<b_{k}<b_{k-2}<...<b_{1}<a_{l-1}<...<a_{1}\leq t_{1},\: w=k,...,s_{2}-1,\:1\leq b_{s_{2}}<...<b_{w+2}<b_{w}^{'}<b_{w-1}^{'}<...<b_{k}^{'}<b_{k-1}<b_{k-2}...<b_{1}\leq t_{2},\: b_{w-1}^{'}\leq b_{w+1},\: b_{r}^{'}\leq b_{r+2}<b_{r+1},\: r=k,...,l-2\}),$\tabularnewline
\end{tabular} \begin{tabular}{>{\centering}p{0.22in}>{\raggedright}p{4.5in}}
(10) & $h_{H^{l,k,q}}=(\{z_{l-1,q}z_{l,q_{l}}...z_{kq_{k}}x_{1a_{1}}...x_{l-2,a_{l-1}}y_{1b_{1}}...y_{k-1b_{k-1}}y_{k+1b_{k+1}}...$\tabularnewline
 & $y_{s_{1}b_{s_{1}}}|\:1\leq l\leq k\leq s_{2},\:1\leq q\leq n,\:1\leq q_{l}<...<q_{k}<b_{s_{1}}<...<b_{k+1}<b_{k-1}<...<b_{1}<a_{l-2}<...<a_{1}\leq t_{1},\: q_{l}<a_{l-1}\leq q<b_{1}\})$,\tabularnewline
\end{tabular} \begin{tabular}{>{\centering}p{0.22in}>{\raggedright}p{4.5in}}
(11) & $h_{I^{l,k,q}}=(\{z_{l-1,q}z_{l,q_{l}}...z_{kq_{k}}x_{1a_{1}}...x_{l-2,a_{l-1}}y_{1b_{1}}...y_{k-1b_{k-1}}y_{kb_{k}}$\tabularnewline
 & $y_{k+1b_{k+1}}...y_{s_{1}b_{s_{1}}}|\:1\leq l\leq k\leq s_{2},\:1\leq q\leq n,\:1\leq q_{l}<...<q_{k}<b_{s_{1}}<...<b_{k+1}<b_{k-1}<b_{k}<b_{k-2}<...<b_{1}<a_{l-2}<...<a_{1}\leq t_{1},\: q_{l}<a_{l-1}\leq q<b_{1}\})$,\tabularnewline
\end{tabular} \begin{tabular}{>{\centering}p{0.22in}>{\raggedright}p{4.5in}}
(12) & \raggedright{}$h_{^{k,w}I^{l,k,q}}=(\{z_{l-1,q}z_{l,q_{l}}...z_{kq_{k}}x_{1a_{1}}...x_{l-2,a_{l-1}}y_{1b_{1}}...y_{kb_{k}}y_{kb_{k}^{'}}...y_{wb_{w}}y_{wb_{w}^{'}}$\tabularnewline
 & \raggedright{}$y_{w+1b_{w+1}}...y_{s_{1}b_{s_{1}}}|\:1\leq l\leq k\leq s_{2},\:1\leq q\leq n,\:1\leq q_{l}<...<q_{k}<b_{s_{1}}<...<b_{k+1}<b_{k-1}<b_{k}<b_{k-2}<...<b_{1}<a_{l-2}<...<a_{1}\leq t_{1},\: q_{l}<a_{l-1}\leq q<b_{1},\: w=k,...,s_{2}-1,\:1\leq b_{s_{2}}<...<b_{w+2}<b_{w}^{'}<b_{w-1}^{'}<...<b_{k}^{'}<b_{k-1}<b_{k-2}...<b_{1}\leq t_{2},\: b_{w-1}^{'}\leq b_{w+1},\: b_{r}^{'}\leq b_{r+2}<b_{r+1},\: r=k,...,l-2\}).$\tabularnewline
\end{tabular}
\par\end{flushleft}

\medskip{}

\section{ALEXANDER DUAL IDEALS}

From Corollary \ref{ininital}, we see that the initial ideal of $\mathcal{K}$
is generated by square free monomials. We know that an ideal generated
by square free monomials defines a Stanley-Reisner ring. Hence we
can find the Alexander dual ideal of this ideal, {[}B-H{]}. We recall
the definition of Alexander dual ideal.
\begin{defn}
\textcolor{black}{If $I$ is an ideal of $R=k[x_{1},...,x_{n}]$ generated
by square-free monomials $(f_{1},...,f_{l})$, then the Alexander
dual ideal}\textcolor{red}{{} }\textcolor{black}{$I^{*}$ of $I$ is
$\cap_{i}P_{f_{i}}$, where for any square-free monomial $f=x_{i_{1}}\cdots x_{i_{r}}$,
$\:$$P_{f}=(x_{i_{1}},...,x_{i_{r}}).$ }
\end{defn}
From Corollary \ref{ininital}, we see that each summand of the initial
ideal has a similar structure as the ideal in the following lemma.
Hence we find the Alexander dual ideal of this ideal first. 
\begin{lem}
\label{ADX} Let $R=k[X]$, where $X=[x_{ij}]$, $i=1,...,m,\mbox{ }j=1,...,n$
and $m\leq n$. Let $I$ be the ideal generated by $\{x_{1a_{1}}x_{2a_{2}}x_{3a_{3}}...x_{ma_{m}}\}$
with $1\leq a_{1}<a_{2}<...<a_{l}\leq a_{l+1}<...<a_{m}\leq n$ for
some $1\leq l\leq m-1$. Then $I^{*}$, the Alexander dual ideal of
$I$, is generated by \begin{eqnarray*}
\{\prod_{i_{1}=1}^{k_{1}}x_{1i_{1}}\prod_{i_{2}=k_{1}+2}^{k_{2}}x_{2i_{2}}\prod_{i_{3}=k_{2}+2}^{k_{3}}x_{3i_{3}}...\prod_{i_{l}=k_{l-1}+2}^{k_{l}}x_{li_{l}}\\
\prod_{i_{l+1}=k_{l}+1}^{k_{l+1}}x_{l+1i_{l+1}}\prod_{i_{l+2}=k_{l+1}+2}^{k_{l+2}}x_{l+2i_{l+2}}...\prod_{i_{m}=k_{m-1}+2}^{n}x_{mi_{m}}\}\end{eqnarray*}
where $0\leq k_{1}<k_{2}<k_{3}...<k_{l}\leq k_{l+1}<...<k_{m-1}<n$.\end{lem}
\begin{proof}
Without lost of generality, we may assume $l=1$. Induction on $m$,
we consider $m=2$ and $m=3$ first. When $m=2$, $I=(\{x_{1a_{1}}x_{2a_{2}}|$
$1\leq a_{1}\leq a_{2}\leq n\}$). Now \begin{eqnarray*}
I^{*} & = & \cap_{1\leq a_{1}\leq a_{2}\leq n}(x_{1a_{1}},\, x_{2a_{2}})=\cap_{1\leq a_{1}\leq n}(\cap_{a_{1}\leq a_{2}\leq n}(x_{1a_{1}},\, x_{2a_{2}}))\\
 & = & \cap_{1\leq a_{1}\leq n}(x_{1a_{1}},\prod_{i_{2}=a_{1}}^{n}x_{2i_{2}})=(\{\prod_{i_{1}=1}^{k_{1}}x_{1i_{1}}\prod_{i_{2}=k_{1}+1}^{n}x_{2i_{2}}|\:0\leq k_{1}\leq n\}).\end{eqnarray*}
When $m=3$, we have $I=(\{x_{1a_{1}}x_{2a_{2}}x_{3a_{3}}|$ $1\leq a_{1}\leq a_{2}<a_{3}\leq n\})$
Now \begin{eqnarray*}
I^{*} & = & \cap_{1\leq a_{1}\leq a_{2}<a_{3}<n}(x_{1a_{1}},\, x_{2a_{2}},x_{3a_{3}})\\
 & = & \cap_{1\leq a_{1}<n}(\cap_{a_{1}\leq a_{2}<n}(\cap_{a_{2}<a_{3}\leq n}(x_{1a_{1}},\, x_{2a_{2}},x_{3a_{3}})))\\
 & = & \cap_{1\leq a_{1}<n}(\cap_{a_{1}\leq a_{2}<n}(x_{1a_{1}},x_{2a_{2}},\prod_{i_{3}=a_{2}+2}^{n}x_{3i_{3}})\\
 & = & \cap_{1\leq a_{1}<n}(x_{1a_{1}},\{\prod_{i_{2}=a_{1}}^{k_{2}}x_{2i_{2}}\prod_{i_{3}=k_{2}+2}^{n}x_{3i_{3}}|\: a_{1}-1\leq k_{2}<n\})\\
 & = & (\{\prod_{i_{1}=1}^{k_{1}}x_{1i_{1}}\prod_{i_{2}=k_{1}+1}^{k_{2}}x_{2i_{2}}\prod_{i_{3}=k_{2}+2}^{n}x_{3i_{3}}|\:0\leq k_{1}\leq k_{2}<n\}).\end{eqnarray*}
 When $m>3$, we have \begin{eqnarray*}
I^{*} & = & \cap_{1\leq a_{1}\leq a_{2}<a_{3}<...<a_{m}\leq n}(x_{1a_{1}},\, x_{2a_{2}},...,\, x_{ma_{m}})\\
 & = & \cap_{1\leq a_{1}\leq a_{2}<a_{3}<...<a_{m}\leq n}(\cap_{a_{1}\leq a_{2}<a_{3}<...<a_{m}\leq n}(x_{1a_{1}},\, x_{2a_{2}},...,\, x_{ma_{m}}))\\
 & = & \cap_{1\leq a_{1}\leq a_{2}<a_{3}<...<a_{m}\leq n}(x_{1a_{1}},\{\prod_{i_{2}=a_{1}}^{k_{2}}x_{2i_{2}}\prod_{i_{3}=k_{2}+2}^{k_{3}}x_{3i_{3}}...\prod_{i_{m}=k_{m-1}+2}^{n}x_{mi_{m}}|\\
 &  & \; a_{1}-1\leq k_{2}<k_{3}<...<k_{m-1}<n\})\\
 & = & (\{\prod_{i_{1}=1}^{k_{1}}x_{1i_{1}}\prod_{i_{2}=k_{1}+1}^{k_{2}}x_{2i_{2}}\prod_{i_{3}=k_{2}+2}^{k_{3}}x_{3i_{3}}...\prod_{i_{m}=k_{m-1}+2}^{n}x_{mi_{m}}|\\
 &  & \:1\leq k_{1}\leq k_{2}<k_{3}<...<k_{m-1}<n\}),\end{eqnarray*}
where the third equality comes from the induction.
\end{proof}
We obtain the Alexander dual ideal of $\mbox{in}(\mathcal{L})$. 

\medskip{}

\begin{lem}
\label{ADL}\begin{eqnarray*}
(\mathrm{in}(\mathcal{L}))^{*} & = & (h_{X})^{*}\cap(h_{Y})^{*}\cap(h_{g})^{*}\cap(h_{f})^{*}\cap(h_{U})^{*}\cap(h_{W})^{*}\cap(h_{W^{p,q,v}})^{*}\cap\\
 &  & (h_{V})^{*}\cap(h_{V^{l,k,w}})^{*}\cap(h_{H})^{*}\cap(h_{I})^{*}\cap(h_{^{k,w}I^{l,k,q}})^{*}\end{eqnarray*}
 where\linebreak{}
\begin{tabular}{c>{\raggedright}p{5.6in}}
$\mathrm{(1)}$ & \raggedright{}\[
(h_{X})^{*}=(\{\prod_{a_{1}=A_{2}+2}^{t_{1}}x_{1a_{1}}\prod_{a_{2}=A_{3}+2}^{A_{2}}x_{2a_{2}}...\prod_{a_{s_{1}}=1}^{A_{s_{1}}}x_{s_{1}a_{s_{1}}}|\:0\leq A_{s_{1}}<...<A_{2}<t_{1}\}),\qquad\qquad\qquad\qquad\qquad\qquad\qquad\qquad\]
 \tabularnewline
\end{tabular} \begin{tabular}{c>{\raggedright}p{5.6in}}
$(2)$ & \[
(h_{Y})^{*}=(\{\prod_{b_{1}=B_{2}+2}^{t_{2}}y_{1b_{1}}\prod_{b_{2}=B_{3}+2}^{B_{2}}y_{2b_{2}}...\prod_{b_{s_{2}}=1}^{B_{s_{2}}}y_{s_{2}b_{s_{2}}}|\:0\leq B_{s_{2}}<...<B_{2}<t_{2}\}),\qquad\qquad\qquad\qquad\qquad\qquad\qquad\qquad\qquad\qquad\qquad\]
\tabularnewline
$(3)$ & \raggedright{}\[
(h_{g})^{*}=(\{\prod_{(i,j)=(1,1)}^{(u_{1},u_{2})}z_{ij}\prod_{(l,k)=(u_{1},u_{2})\uplus2}^{(m,n)}x_{lk}|\:(0,0)\leq(u_{1},u_{2})<(m,n)\}),\;\qquad\qquad\qquad\qquad\qquad\qquad\qquad\qquad\]
\tabularnewline
\end{tabular} \begin{tabular}{l>{\raggedright}p{5.6in}}
$(4)$ & \raggedright{}\[
(h_{f})^{*}=\bigcap_{l=1}^{s_{2}}\bigcap_{k=l}^{s_{2}}(\{\prod_{q_{l}=1}^{Q_{l}}z_{l,q_{l}}..\mbox{.}\prod_{q_{k}=Q_{k-1}+2}^{Q_{k}}z_{kq_{k}}\prod_{a_{1}=A_{2}+2}^{t_{1}}x_{1a_{1}}...\prod_{a_{l-1}=B_{1}+2}^{A_{l-1}}x_{l-1a_{l-1}}\qquad\qquad\qquad\qquad\qquad\qquad\qquad\qquad\qquad\]
\[
\qquad\prod_{b_{1}=B_{2}+2}^{B_{1}}y_{1b_{1}}...\prod_{b_{k-1}=B_{k+1}+2}^{B_{k-1}}y_{k-1b_{k-1}}\prod_{b_{k+1}=B_{k+2}+2}^{B_{k+1}}y_{k+1b_{k+1}}...\qquad\qquad\qquad\qquad\qquad\qquad\]
\[
\qquad\prod_{b_{s_{1}}=Q_{k}+2}^{B_{s_{1}}}y_{s_{1}b_{s_{1}}}|\:0\leq Q_{l}<...<Q_{k}<B_{s_{1}}<...<B_{k+1}<B_{k-1}<...\qquad\qquad\qquad\qquad\qquad\qquad\]
\[
\qquad\qquad\qquad\qquad\qquad\qquad\qquad\qquad<B_{1}<A_{l-1}<...<A_{1}<t_{1}\}),\qquad\qquad\qquad\qquad\qquad\qquad\]
\tabularnewline
\end{tabular} \begin{tabular}{>{\centering}p{0.18in}>{\raggedright}p{5.6in}}
$(5)$ & \[
(h_{U})^{*}=\bigcap_{p=1}^{s_{1}-1}(\{\prod_{q_{p}=A_{p}+1}^{n}z_{p,q_{p}}\prod_{a_{1}=A_{2}+2}^{t_{1}}x_{1a_{1}}...\prod_{\mbox{min}(A_{p},A_{p+1})+2}^{A_{p-1}}x_{p-1,a_{p-1}}\qquad\qquad\qquad\qquad\]
\[
\qquad\prod_{a_{p}=1}^{A_{p}}x_{pa_{p}}\prod_{a_{p+1}=A_{p+2}+2}^{A_{p+1}}y_{p+1a_{p+1}}...\prod_{a_{s_{1}}=1}^{A_{s_{1}}}y_{s_{1}a_{s_{1}}}|0\leq A_{s_{1}}<...<A_{p+1}<\qquad\qquad\qquad\qquad\qquad\]
\[
A_{p-1}<...<A_{2}<t_{1},0\leq A_{p}<A_{p-1}\}),\qquad\qquad\qquad\qquad\]
\tabularnewline
\end{tabular} \begin{tabular}{>{\centering}p{0.15in}>{\centering}p{5.6in}}
$(6)$ & \[
(h_{W})^{*}=\bigcap_{p=1}^{s_{1}-1}\bigcap_{i=1}^{p}(\{\prod_{q=A_{p}+1}^{n}z_{pq}\prod_{a_{1}=A_{2}+2}^{t_{1}}x_{1a_{1}}...\prod_{a_{p-1}=B_{i+1}+2}^{A_{p-1}}x_{p-1,a_{p-1}}\prod_{a_{p}=1}^{A_{p}}x_{pa_{p}}\qquad\qquad\qquad\qquad\qquad\qquad\]
\[
\prod_{b_{1}=B_{2}}^{t_{2}}y_{1b_{1}}...\prod_{b_{p-1}=B_{p}+2}^{B_{p-1}}y_{p-1b_{p-1}}\prod_{b_{p}=B_{p+2}+2}^{B_{p}}y_{pb_{p}}\prod_{b_{p+1}=A_{p-1}+1}^{B_{p+1}}y_{p+1b_{p+1}}...\qquad\qquad\]
\[
\prod_{1}^{B_{S_{1}}}y_{s_{1}b_{s_{1}}}|\:0\leq B_{s_{1}}<...<B_{s_{2}+1}<A_{p}<...<A_{2}<t_{1},\:1\leq B_{s_{1}}<B_{s_{1}-1}\qquad\qquad\]
\[
<...<B_{p+2}<B_{p}<B_{p-1}<...<B_{i+1}<B_{p+1}<B_{i}<...<B_{2}<t_{2},\:\qquad\qquad\]
\[
1\leq B_{s_{1}}<...<B_{p+2}<B_{p}<...<B_{i+1}<A_{p-1}<...<A_{2}<t_{1}\}),\]
\tabularnewline
\end{tabular} \begin{tabular}{>{\centering}p{0.17in}>{\raggedright}p{5.6in}}
$(7)$ & \raggedright{}\[
(h_{W^{p,q,v}})^{*}=\bigcap_{p=1}^{s_{2}-1}\bigcap_{v=p+1}^{s_{2}-1}\bigcap_{i=1}^{p}(\{\prod_{q_{p}=A_{p}+1}^{s_{1}}z_{pq_{p}}\prod_{a_{1}=A_{2}+2}^{A_{1}}x_{1a_{1}}...\prod_{a_{p}=B_{p+1}+2}^{A_{p}}x_{pa_{p}}\qquad\qquad\qquad\qquad\qquad\qquad\]
\[
\prod_{b_{1}=B_{2}+2}^{t_{2}}y_{1b_{1}}...\prod_{b_{p-1}=B_{p+1}^{'}+2}^{B_{p-1}}y_{p-1b_{p-1}}\prod_{b_{p}=A_{p-1}+1}^{B_{p}}y_{pb_{p}}\prod_{b_{p+1}=A_{p-1}+1}^{B_{p+1}}y_{p+1b_{p+1}}\qquad\qquad\qquad\qquad\qquad\qquad\]
\[
\prod_{b_{p+1}^{'}=B_{p+2}^{'}+2}^{B_{p+1}^{'}}y_{p+1b_{p+1}^{'}}\prod_{b_{p+2}=B_{p+3}+2}^{B_{p+2}}y_{p+2b_{p+2}}\prod_{b_{p+2}^{'}=B_{p+3}^{'}+2}^{B_{p+2}^{'}}y_{p+2b_{p+2}^{'}}...\qquad\qquad\qquad\qquad\qquad\qquad\]
\[
\prod_{b_{v}=\mbox{min}(B_{v+1}+2,B_{v-2}^{'}+1)}^{B_{v}}y_{vb_{v}}\prod_{b_{v}^{'}=B_{v+2}+2}^{B_{v}^{'}}y_{vb_{v}^{'}}\prod_{b_{v+1}=B_{v-1}^{'}+1}^{B_{v+1}}y_{v+1b_{v+1}}...\qquad\qquad\qquad\qquad\qquad\qquad\]
\[
\prod_{b_{s_{1}}=A_{s_{2}+1}+2}y_{s_{1}b_{s_{1}}}|\:0\leq B_{s_{1}}<...<B_{s_{2}+1}<A_{p}<...<A_{2}<t_{1},\: B_{s_{2}+1}<\qquad\qquad\qquad\qquad\qquad\qquad\]
\[
B_{s_{2}}<...<B_{v+2}<B_{v}^{'}<...<B_{p+1}^{'}<B_{p}<...<B_{i+1}<A_{p-1}\leq B_{p+1}\qquad\qquad\qquad\qquad\qquad\qquad\]
\[
<B_{i}<...<B_{2}<t_{2},\: B_{s_{1}}<....<B_{s_{2}+1}<B_{s_{2}}<...<B_{v+2}<B_{v+1}<\qquad\qquad\qquad\qquad\qquad\qquad\]
\[
B_{v}<B_{v-1}<...<B_{p+2}<A_{p}<A_{p-1}\leq B_{p+1},\: B_{r}^{'}\leq B_{r+2},\: r=p,...,v-2\}),\qquad\qquad\qquad\qquad\qquad\qquad\]
\tabularnewline
\end{tabular} \begin{tabular}{c>{\centering}p{5.6in}}
$(8)$ & \[
(h_{V})^{*}=\bigcap_{l=1}^{s_{1}-1}\bigcap_{k=l}^{s_{1}}(\{\prod_{q_{l}=1}^{Q_{l}}z_{lq_{l}}...\prod_{q_{k}=Q_{k-1}+2}^{Q_{k}}z_{kq_{k}}\prod_{a_{1}=A_{2}+2}^{t_{1}}x_{1a_{1}}...\prod_{a_{l-1}=B_{1}+2}^{A_{l-1}}x_{l-1a_{l-1}}\qquad\qquad\qquad\qquad\qquad\qquad\]
\[
\prod_{b_{1}=B_{2}+2}^{B_{1}}y_{1b_{1}}...\prod_{b_{k-1}=B_{k+1}+2}^{B_{k-1}}y_{k-1b_{k-1}}\prod_{b_{k}=B_{k-1}+2}^{B_{k}}y_{kb_{k}}\prod_{b_{k+1}=B_{k+2}+1}^{B_{k+1}}y_{k+1b_{k+1}}...\qquad\qquad\qquad\qquad\qquad\qquad\]
\[
\prod_{b_{s_{1}}=Q_{k}+2}^{B_{s_{1}}}y_{s_{1}b_{s_{1}}}|\:0\leq Q_{l}<Q_{l+1}<...<Q_{k}<B_{s_{1}}<...<B_{k+2}<B_{k}<\qquad\qquad\qquad\qquad\qquad\qquad\]
\[
B_{k-1}<...<B_{1}<A_{l-1}<...<A_{2}<t_{1},\, B_{k-1}\leq t_{2}-k+1\}),\qquad\qquad\qquad\qquad\qquad\qquad\]
\tabularnewline
\end{tabular} \begin{tabular}{c>{\raggedright}p{5.6in}}
$(9)$ & \[
(h_{V^{l,k,w}})^{*}=\bigcap_{l=1}^{s_{2}-1}\bigcap_{k=l}^{s_{2}}\bigcap_{w=k}^{s_{1}}(\{\prod_{q_{l}=1}^{Q_{l}}z_{lq_{l}}...\prod_{q_{k}=Q_{k-1}+2}^{Q_{k}}z_{kq_{k}}\prod_{a_{1}=A_{2}+2}^{t_{1}}x_{1a_{1}}...\qquad\qquad\qquad\qquad\qquad\qquad\]
\[
\prod_{a_{l-1}=B_{1}+2}^{A_{l-1}}x_{l-1a_{l-1}}\prod_{b_{1}=B_{2}+2}^{B_{1}}y_{1b_{1}}...\prod_{b_{k-1}=B_{k}+2}^{B_{k-1}}y_{k-1b_{k-1}}\prod_{b_{k}=B_{k+1}^{'}+2}^{B_{k}}y_{kb_{k}}\qquad\qquad\qquad\qquad\qquad\qquad\]
\[
\prod_{b_{k+1}=\mbox{min}(B_{k+2}+2,B_{k-1}+1)}^{t_{2}-k+1}y_{k+1b_{k+1}}\prod_{b_{k+1}^{'}=B_{k+2}^{'}+2}^{B_{k+1}^{'}}y_{k+1b_{k+2}^{'}}...\qquad\qquad\qquad\qquad\qquad\qquad\]
\[
\prod_{b_{w-1}=\mbox{min}(B_{w-2}^{'},B_{w})+2}^{B_{w-1}}y_{w-1b_{w-1}}\prod_{b_{w-1}^{'}=B_{w+1}+2}^{B_{w-1}^{'}}y_{w-1b_{w-1}^{'}}\prod_{b_{w}=B_{w+1}+2}^{B_{w}}y_{wb_{w}}\qquad\qquad\qquad\qquad\qquad\qquad\]
\[
\prod_{b_{w}^{'}=B_{w+2}+2}^{B_{w}^{'}}y_{wb_{w}^{'}}\prod_{b_{w+1}=B_{w}^{2}+2}^{B_{w+1}}y_{w+1b_{w+1}}\prod_{b_{w+2}=B_{w+3}+2}^{B_{w+2}}y_{w+2b_{w+2}}...\prod_{b_{s_{1}}=Q_{k}+2}^{B_{s_{1}}}y_{s_{1}b_{s_{1}}}|\:\qquad\qquad\qquad\qquad\qquad\qquad\]
\[
0\leq Q_{l}<Q_{l+1}<...<Q_{k}<B_{s_{1}}<...<B_{t+2}<B_{t}^{2}<B_{t-1}^{2}<...<B_{k+1}^{2}\qquad\qquad\qquad\qquad\qquad\qquad\]
\[
<B_{k}<B_{k-1}<...<B_{1}<A_{l-1}<...<A_{2}<t_{1},\: B_{w-1}^{2}\leq B_{w+1},\mbox{ }\:\:\:\qquad\qquad\qquad\qquad\qquad\qquad\]
\[
\: B_{r}^{2}\leq B_{r+1}^{1}<B_{r}^{1},\: r=k+1,...,w-2,\:0\leq Q_{p+1}<Q_{p+2}<....<Q_{k}<B_{s_{1}}\qquad\qquad\qquad\qquad\qquad\qquad\]
\[
<...<\: B_{k+1}<B_{k-1}<...<B_{1}<A_{p-1}<A_{p-2}<...<A_{1}<t_{1}\}),\qquad\qquad\qquad\]
\tabularnewline
\end{tabular} \begin{tabular}{>{\centering}p{0.27in}>{\raggedright}p{5.6in}}
$(\mathrm{1}0)$ & \raggedright{}\[
(h_{H^{l,k,q}})^{*}=\bigcap_{l=1}^{s_{2}}\bigcap_{k=l}^{s_{2}}(\{\prod_{q=A_{l-1}+1}^{Q}z_{l-1,q}\prod_{q_{l}=1}^{Q_{l}}z_{l,q_{l}}...\prod_{q_{k}=Q_{k-1}+2}^{Q_{k}}z_{kq_{k}}\prod_{a_{1}=A_{2}+2}^{t_{1}}x_{1a_{1}}...\qquad\qquad\qquad\qquad\qquad\qquad\]
\[
\prod_{a_{l-1}=B_{1}+2}^{A_{l-1}}x_{l-2,a_{l-1}}\prod_{b_{1}=\mbox{min}(B_{2},Q)+2}^{B_{1}}y_{1b_{1}}...\prod_{b_{k-1}=B_{k+1}+2}^{B_{k-1}}y_{k-1b_{k-1}}\qquad\qquad\qquad\qquad\qquad\qquad\]
\[
\prod_{b_{k+1}=B_{k+2}+2}^{B_{k+1}}y_{k+1b_{k+1}}...\prod_{b_{s_{1}}=Q_{k}+2}^{B_{s_{1}}}y_{s_{1}b_{s_{1}}}|\:1\leq Q_{l}<...<Q_{k}<B_{s_{1}}<...<\qquad\qquad\qquad\qquad\qquad\qquad\]
$B_{k+1}<B_{k-1}<...<B_{1}<A_{l-2}<...<A_{2}<t_{1},\: Q_{l}<A_{l-1}\leq Q<B_{1}\}),\qquad\qquad\qquad\qquad\qquad\qquad$\tabularnewline
\end{tabular} \begin{tabular}{c>{\centering}p{5.6in}}
$(\mathrm{1}1)$ & \[
(h_{I^{l,k,q}})^{*}=\bigcap_{l=1}^{s_{2}}\bigcap_{k=l}^{s_{2}}(\{\prod_{q=A_{l-1}+1}^{Q}z_{l-1,q}\prod_{q_{l}=1}^{Q_{l}}z_{l,q_{l}}...\prod_{Q_{k}=Q_{k-1}+2}^{Q_{k}}z_{kq_{k}}\prod_{a_{1}=A_{2}+2}^{A_{1}}x_{1a_{1}}...\qquad\qquad\qquad\qquad\qquad\qquad\]
\[
\prod_{a_{l-2}=B_{1}+2}^{A_{l-2}}x_{l-2,a_{l-1}}\prod_{b_{1}=\mbox{min}(B_{2},Q)+2}^{B_{1}}y_{1b_{1}}...\prod_{b_{k-1}=B_{k+1}+2}^{B_{k-1}}y_{k-1b_{k-1}}\prod_{b_{k}=B_{k-1}+2}^{B_{k}}y_{kb_{k}}\qquad\qquad\qquad\qquad\qquad\qquad\]
\[
\prod_{b_{k+1}=B_{k+2}+2}^{B_{k+1}}y_{k+1b_{k+1}}...\prod_{b_{s_{1}}=Q_{k}+2}^{B_{s_{1}}}y_{s_{1}b_{s_{1}}}|\:\:1\leq Q_{l}<...<Q_{k}<B_{s_{1}}<...<\qquad\qquad\qquad\qquad\qquad\qquad\]
\[
B_{k+1}<...<B_{k+1}<B_{k-1}<B_{k}<B_{k-2}<...<B_{1}<A_{l-2}<...<A_{2}<t_{1},\quad\qquad\qquad\qquad\]
\[
Q_{l}<A_{l-1}\leq Q<B_{1}\}),\]
\tabularnewline
\end{tabular} \begin{tabular}{l>{\raggedright}p{5.6in}}
$(\mathrm{1}2)$ & \[
(h_{^{k,w}I^{l,k,q}})^{*}=\bigcap_{l=1}^{s_{2}}\bigcap_{k=l}^{s_{2}}\bigcap_{w=k}^{s_{2}}(\{\prod_{q=A_{l-1}+1}^{Q}z_{l-1,q}\prod_{q_{l}=1}^{Q_{l}}z_{l,q_{l}}...\prod_{q_{k}=Q_{k-1}+2}^{Q_{k}}z_{kq_{k}}\qquad\qquad\qquad\qquad\qquad\qquad\]
\[
\prod_{a_{1}=A_{2}+2}^{t_{1}}x_{1a_{1}}...\prod_{a_{l-2}=B_{1}+2}^{A_{l-2}}x_{l-2,a_{l-1}}\prod_{b_{1}=\mbox{min}(Q,B_{2})+2}^{B_{1}}y_{1b_{1}}...\qquad\qquad\qquad\qquad\qquad\qquad\]
\[
\prod_{b_{k-1}=B_{k+1}+2}^{B_{k-1}}y_{k-1b_{k-1}}\prod_{b_{k}=B_{k-1}+2}^{B_{k}}y_{kb_{k}}\prod_{b_{k}^{'}=B_{k+1}^{'}+2}^{B_{k}^{'}}y_{kb_{k}^{'}}...\prod_{b_{w}=B_{w+1}+2}^{B_{w}}y_{wb_{w}}\qquad\qquad\qquad\qquad\qquad\qquad\]
\[
\prod_{b_{w}^{'}=B_{w+2}+2}^{B_{w}^{'}}y_{wb_{w}^{'}}...\prod_{b_{w+1}=B_{w-1}^{'}+1}^{B_{w+1}}y_{w+1b_{w+1}}\prod_{b_{s_{1}}=Q_{k}+2}^{B_{s_{1}}}y_{s_{1}b_{s_{1}}}|\:\:1\leq Q_{l}<...<Q_{k}<\qquad\qquad\qquad\qquad\qquad\qquad\]
\[
B_{s_{1}}<...<B_{k+1}<B_{k-1}<B_{k}<B_{k-2}<...<B_{1}<A_{l-2}<...<A_{2}<t_{1},\qquad\qquad\qquad\qquad\qquad\qquad\]
\[
\: Q_{l}<A_{l-1}\leq Q<B_{1},\:1\leq B_{s_{2}}<...<B_{w+2}<B_{w}^{'}<B_{w-1}^{'}<...<B_{k}^{'}<\qquad\qquad\qquad\qquad\qquad\qquad\]
\[
B_{k-1}<B_{k-2}...<B_{2}<t_{2},\: B_{w-1}^{'}\leq B_{w+1},\: B_{r}^{'}\leq B_{r+2}<B_{r+1},\: k\leq r\leq l-2\}).\qquad\qquad\qquad\qquad\qquad\qquad\]
\tabularnewline
\end{tabular}\end{lem}
\begin{proof}
This follows from Lemma \ref{ADX}.
\end{proof}
Having the Alexander dual ideal of $\mbox{in}(\mathcal{L})$, we can
use Theorem \ref{ER} to show that $\mbox{in}(\mathcal{L})$ is Cohen-Macaulay
once we show that the Alexander dual ideal has a linear free resolution.
We recall the definition of a linear free solution and the regularity
of an ideal.
\begin{defn}
\textcolor{black}{(a) Let \[
\mathbf{F}:...\longrightarrow F_{i}\longrightarrow F_{i-1}\longrightarrow...\longrightarrow F_{0}\]
be a minimal homogeneous free resolution of an ideal $I$ in a ring
$R=k[x_{1},...,x_{n}]$ with $F_{i}=\oplus_{j}R(-a_{ij})$. We say
$I$ has a linear free resolution if $a_{ij}=a_{i}$ and $a_{i+1}=a_{i}+1$.}
\end{defn}
\textcolor{black}{(b) The regularity of $I$ is defined as $\mbox{reg}(I)=\mbox{\mbox{max}}_{i,j}\{a_{ij}-i\}.$}

\medskip{}

\begin{fact}
\textcolor{black}{\label{linearAndreg} If all the minimal homogeneous
generators of $I$ have the same degree, $d$, then $I$ has a linear
free resolution if and only if $\mathrm{reg}$$(I)=d$.}
\end{fact}
We will show that $(\mbox{in}(\mathcal{L}))^{*}$ is generated in
the same degree $d$ and $\mbox{reg}(\mbox{in}(\mathcal{L}))^{*}=d$.
Before that we show the following result first. The reason we show
this is that we will need the technique of the proof for the case
$(\mbox{in}(\mathcal{L}))^{*}$.
\begin{lem}
\label{REGX} Let $R=k[X]$, where $X=[x_{ij}]$, $i=1,...,m,\mbox{ }j=1,...,n$.
Let $I$ be the ideal generated by $\{x_{1a_{1}}x_{2a_{2}}x_{3a_{3}}...x_{ma_{m}}\}$
with $1\leq a_{1}<a_{2}<...<a_{l}\leq a_{l+1}<...<a_{m}<n$ for some
$1\leq l\leq m-1$. Then $I^{*}$, the Alexander dual ideal of $I$
has a linear free resolution.\end{lem}
\begin{proof}
From Lemma \ref{ADX}, we see that $I^{*}$ is generated by elements
of degree $n-(m-2)$, denoted by $\mbox{d}(I^{*})$. Using Fact \ref{linearAndreg},
it's sufficient to show that reg($I^{*})=\mbox{d}(I^{*})=n-(m-2)$.
We will induct on $n$ to show that there is a linear filtration on
$I^{*}$. 

We write down $I^{*}$ first, \[
I^{*}=(\{\prod_{i_{1}=1}^{k_{1}}x_{1i_{1}}\prod_{i_{2}=k_{1}+2}^{k_{2}}x_{2i_{2}}...\prod_{i_{l}=k_{l-1}+2}^{k_{l}}x_{ii_{l}}\prod_{i_{l+1}=k_{l}+1}^{k_{l+1}}x_{ii_{l+1}}...\prod_{i_{m}=k_{m-1}+2}^{n}x_{mi_{m}}\})\]
 with $0\leq k_{1}<k_{2}<...<k_{l}\leq k_{l+1}<...<k_{m-1}<n.$ When
$m=n$, we need $0\leq k_{1}<k_{2}<...<k_{l}\leq k_{l+1}<...<k_{m-1}<m.$
Without lost of generality, we assume $l=1$. Hence \[
I^{*}=(x_{11}x_{12},\:\{x_{11}x_{ii}|\: i=2,...,m\},\:\{x_{i,i-1}x_{j,j}|\: i=2,...,m,\: j=i,...,m\}).\]
 Now look at $x_{m,m}$, the terms $x_{i,i-1}x_{mm}$ for $i=2,...,m-1$
and $x_{11}x_{m,m}$ are multiple of $x_{m,m}$. Also $x_{i,i-1}x_{j,j}$
for $i=2,...,m-1$, $j=,...,m-1$ is divisible by $x_{i,i-1}$ and
$x_{11}x_{12}$ and $x_{11}x_{ii}$ for $i=2,...,m-1$ is divisible
by $x_{11}$. We can rewrite \[
I^{*}=(\{x_{i,i-1}(x_{i,i},...,x_{m,m})|\: i=2,...,m-1\},\: x_{11}(x_{12},x_{2,2},...,x_{m,m})).\]
Then we have

\begin{eqnarray*}
I^{*}=J_{1} & \subset & (\{x_{i,i-1}(x_{i,i},...,x_{m,m})|\: i=2,...,m-1\},x_{11})=J_{2}\\
 & \subset & (\{x_{i,i-1}(x_{i,i},...,x_{m,m})|\: i=3,...,m-1\},x_{11},x_{21})=J_{3}\\
 & \subset & (\{x_{i,i-1}(x_{i,i},...,x_{m,m})|\: i=4,...,m-1\},x_{11},x_{21},x_{32})=J_{4}\\
 & \subset & ...\\
 & \subset & (\{x_{i,i-1}|\: i=2,...,m-1\},x_{11})=J_{m}.\end{eqnarray*}
$J_{m}$ is generated by a regular sequence of degree 1, hence it
has reg$J_{m}$=1. 

We will show reg$(J_{l+1}/J_{l})=1$ for $l=1,...,m-1$. Then $\mbox{reg}(J_{l})=2=m-(m-2)$
for $l=1,...,m-1$. We write down \[
J_{l+1}=(\{x_{i,i-1}(x_{i,i},...,x_{m,m})|\: i=l+1,...,m-1\},x_{11},x_{21},x_{32},...,x_{l-1,l-2},x_{l,l-1})\]
 and \begin{eqnarray*}
J_{l} & = & (\{x_{i,i-1}(x_{i,i},...,x_{m,m})|\: i=l,...,m-1\},x_{11},x_{21},x_{32},...,x_{l-1,l-2})\\
 & = & (\{x_{i,i-1}(x_{i,i},...,x_{m,m})|\: i=l+1,...,m-1\},\: x_{l,l-1}(x_{l,l},...,x_{mm}),\\
 &  & x_{11},x_{21},x_{32},...,x_{l-1,l-2}).\end{eqnarray*}
Then \begin{eqnarray*}
J_{l+1}/J_{l} & = & (x_{l,l-1})/(x_{l,l-1}\cap(\{x_{i,i-1}(x_{i,i},...,x_{m,m})|\: i=l+1,...,m-1\}),\\
 &  & x_{l,l-1}\cap(x_{11},x_{21},x_{32},...,x_{l-1,l-2}),x_{l,l-1}(x_{l,l},...,x_{mm}))\\
 & = & (x_{l,l-1})/x_{l,l-1}(x_{l,l},...,x_{mm},x_{11},x_{21},x_{32},...,x_{l-1,l-2}).\end{eqnarray*}
Since \begin{eqnarray*}
 & \mbox{reg}((x_{l,l-1})/x_{l,l-1}(x_{l,l},...,x_{mm},x_{11},x_{21},x_{32},...,x_{l-1,l-2}))\\
= & \mbox{reg}(R/(x_{l,l},...,x_{mm},x_{11},x_{21},x_{32},...,x_{l-1,l-2}))+1 & =1,\end{eqnarray*}
we have $\mbox{reg}(J_{l+1}/J_{l})=1$ for all $l=1,...,m-1$. 

For the induction part, we write $I^{*}:=I_{n}^{*}$ when $X$ is
a $m$ by $n$ matrix. We assume $\mbox{reg}(I_{n-1}^{*})=n-1-(m-2)$
and the degree of the generating of $I_{n-1}^{*}$ is $n-1-(m-2)$.
We look at the variable $x_{m,n}$. When $k_{m-1}<n-1$, \[
\{\prod_{i_{1}=1}^{k_{1}}x_{1i_{1}}\prod_{i_{2}=k_{1}+1}^{k_{2}}x_{2i_{2}}\prod_{i_{3}=k_{2}+2}^{k_{3}}x_{3i_{3}}...\prod_{i_{m}=k_{m-1}+2}^{n}x_{m,i_{m}}\}\]
 is divisible by $x_{m,n}$ and we write \begin{eqnarray*}
A= & \{\prod_{i_{1}=1}^{k_{1}}x_{1i_{1}}\prod_{i_{2}=k_{1}+1}^{k_{2}}x_{2i_{2}}\prod_{i_{3}=k_{2}+2}^{k_{3}}x_{3i_{3}}...\prod_{i_{m}=k_{m-1}+2}^{n}x_{m,i_{m}}|\\
 & 0\leq k_{1}\leq k_{2}<k_{3}<...<k_{m-1}=n-2\}\\
= & \{(\prod_{i_{1}=1}^{k_{1}}x_{1i_{1}}\prod_{i_{2}=k_{1}+1}^{k_{2}}x_{2i_{2}}\prod_{i_{3}=k_{2}+2}^{k_{3}}x_{3i_{3}}...\prod_{i_{m}=k_{m-2}+2}^{n-2}x_{m-1,i_{m-1}})x_{m,n}|\\
 & 0\leq k_{1}\leq k_{2}<k_{3}<...<k_{m-1}=n-2\}\\
= & A'x_{m,n}\end{eqnarray*}
 When $k_{m-1}=n-1$, $\prod_{i_{1}=1}^{k_{1}}x_{1i_{1}}\prod_{i_{2}=k_{1}+1}^{k_{2}}x_{2i_{2}}\prod_{i_{3}=k_{2}+2}^{k_{3}}x_{3i_{3}}...\prod_{i_{m}=k_{m-1}+2}^{n}x_{mi_{m}}$
is not divisible by $ $$x_{m,n}$, we have \begin{eqnarray*}
B= & \{\prod_{i_{1}=1}^{k_{1}}x_{1i_{1}}\prod_{i_{2}=k_{1}+1}^{k_{2}}x_{2i_{2}}\prod_{i_{3}=k_{2}+2}^{k_{3}}x_{3i_{3}}...\prod_{i_{m}=k_{m-1}+2}^{n}x_{mi_{m}}|\\
 & 0\leq k_{1}\leq k_{2}<k_{3}<...<k_{m-1}=n-1\}\\
= & \{\prod_{i_{1}=1}^{k_{1}}x_{1i_{1}}\prod_{i_{2}=k_{1}+1}^{k_{2}}x_{2i_{2}}\prod_{i_{3}=k_{2}+2}^{k_{3}}x_{3i_{3}}...\prod_{i_{m}=k_{m-2}+2}^{n-1}x_{m-1,i_{m-1}}|\\
 & 0\leq k_{1}\leq k_{2}<k_{3}<...<k_{m-2}<k_{m-1}=n-1\}\\
= & \{(\prod_{i_{1}=1}^{k_{1}}x_{1i_{1}}\prod_{i_{2}=k_{1}+1}^{k_{2}}x_{2i_{2}}\prod_{i_{3}=k_{2}+2}^{k_{3}}x_{3i_{3}}...\prod_{i_{m}=k_{m-2}+2}^{n-1}x_{m-1,i_{m-1}})\\
 & x_{j,j-1+(n-1-(m-2))-1}|\:0\leq k_{1}\leq k_{2}<k_{3}<...<k_{m-2}<n-2,\:1<j<m\},\\
 & \{(\prod_{i_{1}=1}^{k_{1}}x_{1i_{1}}\prod_{i_{2}=k_{1}+1}^{k_{2}}x_{2i_{2}}\prod_{i_{3}=k_{2}+2}^{k_{3}}x_{3i_{3}}...\prod_{i_{m}=k_{m-2}+2}^{n-1}x_{m-1,i_{m-1}})\\
 & x_{1,1+(n-1-(m-2))-1}|\:0\leq k_{1}\leq k_{2}<k_{3}<...<k_{m-2}<n-2\}.\\
= & A'(\{x_{j,j-1+(n-1-(m-2))-1}|\:1<j<m\},x_{1,n-1-(m-2)}).\end{eqnarray*}
When $k_{m-1}\leq n-3$, we write \begin{eqnarray*}
C & = & \{\prod_{i_{1}=1}^{k_{1}}x_{1i_{1}}\prod_{i_{2}=k_{1}+1}^{k_{2}}x_{2i_{2}}\prod_{i_{3}=k_{2}+2}^{k_{3}}x_{3i_{3}}...\prod_{i_{m}=k_{m-1}+2}^{n}x_{m,i_{m}}|\\
 &  & 0\leq k_{1}\leq k_{2}<k_{3}<...<k_{m-2}<k_{m-1}\leq n-3\}\\
 & = & \{(\prod_{i_{1}=1}^{k_{1}}x_{1i_{1}}\prod_{i_{2}=k_{1}+1}^{k_{2}}x_{2i_{2}}\prod_{i_{3}=k_{2}+2}^{k_{3}}x_{3i_{3}}...\prod_{i_{m}=k_{m-1}+2}^{n-1}x_{m,i_{m}})x_{m,n}|\\
 &  & 0\leq k_{1}\leq k_{2}<k_{3}<...<k_{m-2}<k_{m-1}\leq n-3\}\\
 & = & C'x_{mn}.\end{eqnarray*}
 On the other hand $C$ can be written as\begin{eqnarray*}
C & = & \{(\prod_{i_{1}=1}^{k_{1}}x_{1i_{1}}\prod_{i_{2}=k_{1}+1}^{k_{2}}x_{2i_{2}}\prod_{i_{3}=k_{2}+2}^{k_{3}}x_{3i_{3}}...\prod_{i_{m}=k_{m-1}+2}^{n-1}x_{m,i_{m}})x_{j,j-1+(n-1-(m-2))-1}x_{m,n}|\\
 &  & 0\leq k_{1}\leq k_{2}<k_{3}<...<k_{m-2}<k_{m-1}-1<k_{m}\leq n-3,\:1<j<m\},\end{eqnarray*}
\begin{eqnarray*}
 &  & \{(\prod_{i_{1}=1}^{k_{1}}x_{1i_{1}}\prod_{i_{2}=k_{1}+1}^{k_{2}}x_{2i_{2}}\prod_{i_{3}=k_{2}+2}^{k_{3}}x_{3i_{3}}...\prod_{i_{m}=k_{m-1}+2}^{n-1}x_{m,i_{m}})x_{1,1+(n-1-(m-2))-1}x_{m,n}|\\
 &  & 0\leq k_{1}\leq k_{2}<k_{3}<...<k_{m-2}<k_{m-1}-1<k_{m}\leq n-3\}\\
 & = & C''x_{m,n}(\{x_{j,j-1+(n-1-(m-2))-1}|\:1<j<m\},x_{1,n-1-(m-2)}).\end{eqnarray*}
Notice $I_{n-1}^{*}=(A^{'},C')$ and $\mbox{d}(A')=\mbox{d}(C')$.
We also have \begin{eqnarray*}
A' & = & (\{(\prod_{i_{1}=1}^{k_{1}}x_{1i_{1}}\prod_{i_{2}=k_{1}+1}^{k_{2}}x_{2i_{2}}\prod_{i_{3}=k_{2}+2}^{k_{3}}x_{3i_{3}}...\prod_{i_{m}=k_{m-2}+2}^{n-2}x_{m-1,i_{m-1}})|\\
 &  & 0\leq k_{1}\leq k_{2}<k_{3}<...<k_{m-2}<k_{m-1}=n-2\})\\
 & = & (\{(\prod_{i_{1}=1}^{k_{1}}x_{1i_{1}}\prod_{i_{2}=k_{1}+1}^{k_{2}}x_{2i_{2}}\prod_{i_{3}=k_{2}+2}^{k_{3}}x_{3i_{3}}...\prod_{i_{m}=k_{m-2}+2}^{n-2}x_{m-1,i_{m-1}})x_{j,j-1+(n-2-(m-2))-1}|\\
 &  & 0\leq k_{1}\leq k_{2}<k_{3}<...<k_{m-2}<k_{m-1}=n-3,\:1<j<m\},\\
 &  & \{(\prod_{i_{1}=1}^{k_{1}}x_{1i_{1}}\prod_{i_{2}=k_{1}+1}^{k_{2}}x_{2i_{2}}\prod_{i_{3}=k_{2}+2}^{k_{3}}x_{3i_{3}}...\prod_{i_{m}=k_{m-2}+2}^{n-2}x_{m-1,i_{m-1}})x_{1,1+(n-2-(m-2))-1}|\\
 &  & 0\leq k_{1}\leq k_{2}<k_{3}<...<k_{m-2}<k_{m-1}=n-3\})\\
 & = & A''(\{x_{j,j-1+n-m-1}|\:1<j<m\},x_{1,n-2-(m-2)}).\end{eqnarray*}
Now look at $C'$, we obtain \begin{eqnarray*}
C' & = & (\{\prod_{i_{1}=1}^{k_{1}}x_{1i_{1}}\prod_{i_{2}=k_{1}+1}^{k_{2}}x_{2i_{2}}\prod_{i_{3}=k_{2}+2}^{k_{3}}x_{3i_{3}}...\prod_{i_{m}=k_{m-1}+2}^{n-1}x_{m,i_{m}}|\\
 &  & 0\leq k_{1}\leq k_{2}<k_{3}<...<k_{m-2}<k_{m-1}<n-3\},\\
 &  & \{\prod_{i_{1}=1}^{k_{1}}x_{1i_{1}}\prod_{i_{2}=k_{1}+1}^{k_{2}}x_{2i_{2}}\prod_{i_{3}=k_{2}+2}^{k_{3}}x_{3i_{3}}...\prod_{i_{m}=k_{m-1}+2}^{n-1}x_{m,i_{m}}|\\
 &  & 0\leq k_{1}\leq k_{2}<k_{3}<...<k_{m-2}<k_{m-1}=n-3\})\\
 & = & (\{(\prod_{i_{1}=1}^{k_{1}}x_{1i_{1}}\prod_{i_{2}=k_{1}+1}^{k_{2}}x_{2i_{2}}\prod_{i_{3}=k_{2}+2}^{k_{3}}x_{3i_{3}}...\prod_{i_{m}=k_{m-1}+2}^{n-2}x_{m,i_{m}})x_{m,n-1}|\\
 &  & 0\leq k_{1}\leq k_{2}<k_{3}<...<k_{m-2}<k_{m-1}<n-3\},\\
 &  & A''x_{m,n-1}).\end{eqnarray*}
Hence \begin{eqnarray*}
(A')\cap(C') & = & (A''(\{x_{j,j-1+n-m-1}|\:1<j<m\},x_{1,n-2-(m-2)}))\cap\\
 &  & (\{(\prod_{i_{1}=1}^{k_{1}}x_{1i_{1}}\prod_{i_{2}=k_{1}+1}^{k_{2}}x_{2i_{2}}\prod_{i_{3}=k_{2}+2}^{k_{3}}x_{3i_{3}}...\prod_{i_{m}=k_{m-1}+2}^{n-2}x_{m,i_{m}})x_{m,n-1}|\\
 &  & 0\leq k_{1}\leq k_{2}<k_{3}<...<k_{m-2}<k_{m-1}<n-3\},\\
 &  & A''x_{m,n-1})\\
 & = & (A''(\{x_{j,j-1+n-m-1}|\:1<j<m\},x_{1,n-2-(m-2)}))\cap\\
 &  & (\{(\prod_{i_{1}=1}^{k_{1}}x_{1i_{1}}\prod_{i_{2}=k_{1}+1}^{k_{2}}x_{2i_{2}}\prod_{i_{3}=k_{2}+2}^{k_{3}}x_{3i_{3}}...\prod_{i_{m}=k_{m-1}+2}^{n-2}x_{m,i_{m}})x_{m,n-1}|\\
 &  & 0\leq k_{1}\leq k_{2}<k_{3}<...<k_{m-2}<k_{m-1}<n-3\}),\\
 &  & (A''(\{x_{j,j-1+n-m-1}|\:1<j<m\},x_{1,n-2-(m-2)}))\cap\\
 &  & (A''x_{m,n-1})\\
 & = & (A'x_{m,n-1})\\
 & = & (C'(\{x_{j,j-1+n-m-1}|\:1<j<m\},x_{1,n-2-(m-2)})).\end{eqnarray*}
On the other hand, we have \begin{eqnarray*}
(A')\cap(C'x_{m,n}) & = & (A')\cap(C''x_{m,n}(\{x_{j,j-1+n-m}|\:1<j<m\},x_{1,n-1-(m-2)}))\\
 & \subset & (A'(\{x_{j,j-1+n-m}|\:1<j<m\},x_{1,n-m+1})).\end{eqnarray*}

We look at the filtration: \begin{eqnarray*}
I_{n}^{*} & = & (C,B,A)\\
 & = & (C'x_{m,n},\: A'(\{x_{j,j-1+n-m}|\:1<j<m\},x_{1,n-m+1}),\: A'x_{m,n})\\
 & \subset & (C'x_{m,n},A')\\
 & \subset & (C',A')=I_{n-1}^{*}.\end{eqnarray*}
 We have \begin{eqnarray*}
 & \mbox{reg}((C',A')/(C'x_{m,n},A'))\\
= & \mbox{reg}((C')/((C')\cap(A'),C'x_{m,n}))\\
= & \mbox{reg}((C')/(C'(\{x_{j,j-1+n-m-1}|\:1<j<m\},x_{1,n-2-(m-2)},x_{m,n}))\\
\mbox{=} & \mbox{reg}(R/(\{x_{j,j-1+n-m-1}|\:1<j<m\},x_{1,n-2-(m-2)},x_{mn})+\mbox{deg}(C')\\
= & \mbox{deg}(C')\\
= & n-(m-2)-1.\end{eqnarray*}
By induction hypothesis, we have $\mbox{reg}(C',A')=\mbox{reg}(I_{n-1}^{*})=n-1-(m-2)$,
hence $\mbox{reg}(C'x_{m,n},A')=n-(m-2)$. Similarly \begin{eqnarray*}
 & \mbox{reg}((C'x_{m,n},A')/(C'x_{m,n},\: A'(\{x_{j,j-1+n-m}|\:1<j<m\},x_{1,n-m+1}),\: A'x_{m,n}))\\
= & \mbox{reg}((A')/((A')\cap(C'x_{m,n}),A'(\{x_{j,j-1+n-m}|\:1<j<m\},x_{1,n-m+1}),\: A'x_{m,n}))\\
= & \mbox{reg}((A')/(A'(\{x_{j,j-1+n-m}|\:1<j<m\},x_{1,n-m+1},\: x_{m,n}))\\
= & \mbox{reg}(R/(A'(\{x_{j,j-1+n-m}|\:1<j<m\},x_{1,nm+1},\: x_{m,n})))+\mbox{d}(A')\\
= & \mbox{d}(A')\\
= & n-(m-2)-1.\end{eqnarray*}
Hence $\mbox{reg}(I_{n}^{*})=n-(m-2)$. \end{proof}
\begin{lem}
\label{REGL} The Alexander dual of $\mathrm{in}(\mathcal{L})$, $(\mathrm{in}(\mathcal{L}))^{*}$,
is generated by square free monomials with degree $mn-1+t_{2}-(s_{2}-1)+t_{1}-(s_{1}-1)$
and $\mathrm{reg}(\mathrm{in}(\mathcal{L}))^{*}=mn-1+t_{2}-(s_{2}-1)+t_{1}-(s_{1}-1)$.\end{lem}
\begin{proof}
We prove this lemma by inducting on $n$. Since $s_{1}>s_{2}$, the
generators of the ideal $(\mbox{in}(\mathcal{L}))^{*}$ do not involving
variables $x_{ij}$, $y_{lk}$, $z_{pq}$ when $i,l,p>s_{1}$. Hence
we may assume $s_{1}=m$. When $m=n$, we have $m=s_{1}\leq t_{1}\leq m$.
Then we have \[
(\mbox{in}(\mathcal{L}))^{*}=(h_{X})^{*}\cap(h_{Y})^{*}\cap(h_{g})^{*}\cap(h_{f})^{*}\cap(h_{U})^{*}\cap(h_{W})^{*}\cap(h_{W^{p,q,v}})^{*}.\]
 We write down each component, \[
(h_{X})^{*}=(x_{1m},\: x_{2,m-1},...,\: x_{m,1}),\]
 \[
(h_{f})^{*}=(t_{1,1},y_{2,m},y_{3,m-1},...,y_{m2}),\]
 \begin{eqnarray*}
(h_{Y})^{*} & = & (\{\prod_{b_{1}=B_{2}+2}^{t_{2}}y_{1b_{1}}\prod_{b_{2}=B_{3}+2}^{B_{2}}y_{2b_{2}}...\prod_{b_{s_{2}}=1}^{B_{s_{2}}}y_{s_{2}b_{s_{2}}}|0\leq B_{s_{2}}<...<B_{2}<t_{2}\}),\end{eqnarray*}
\[
(h_{g})^{*}=(\{\prod_{(i,j)=(1,1)}^{(u_{1},u_{2})}z_{ij}\prod_{(l,k)=(u_{1},u_{2})\uplus2}^{(m,n)}x_{lk}|\:(0,0)\leq(u_{1},u_{2})<(m,n)\}),\]
\begin{eqnarray*}
(h_{U})^{*} & = & \bigcap_{p=1}^{m-1}\left(\bigcap_{(p,u_{2})=(p,m)}^{(p,m-(p-1))}(z_{p,u_{2}},x_{1,m},...,x_{p,m-(p-1)},y_{p+1,m-p},...,y_{m,1})\right.\\
 &  & \bigcap_{(p,u_{2})=(p,m-(p-1)-1)}^{(p,1)}(z_{p,u_{2}},x_{1,m},...,x_{p,u_{2}},y_{p+1,m-p+1},...,y_{j,u_{2}+1},\\
 &  & \left.y_{j+1,u_{2}-1},...,y_{m,1})\right),\end{eqnarray*}
\begin{eqnarray*}
(h_{W})^{*} & = & \bigcap_{p=1}^{m-1}\bigcap_{(p,u_{2})=(p,m)}^{(p,m-(p-1))}(z_{p,u_{2}},x_{1,m},...,x_{p,m-(p-1)},y_{1,m},...,y_{p-1,m-p+2},\\
 &  & y_{p,m-p},y_{p+1,m-p+1},y_{p+2,m-p-1},...,y_{m,1}),\end{eqnarray*}
\begin{eqnarray*}
(h_{W^{p,q,v}})^{*} & = & \bigcap_{p=1}^{m-1}\bigcap_{v=p+1}^{m-1}\bigcap_{(p,u_{2})=(p,m)}^{(p,m-(p-1))}(z_{p,u_{2}},x_{1,m},...,x_{p,m-(p-1)},y_{1,m},...,y_{p-1,m-p+2},\\
 &  & y_{p.m-p},y_{p+1,m-p+1},y_{p+1,m-p-1},y_{p+2,m-p},y_{p+2,m-p-2},...,\\
 &  & y_{v,m-(v-1)-1},y_{v,m-(v-1)+1},y_{v+1,m-(v-1)},...,y_{m,1}).\end{eqnarray*}
\\
Then we have\begin{eqnarray*}
(\mbox{in}(\mathcal{L}))^{*} & = & (\{\prod_{(i,j)=(1,1)}^{(u_{1},u_{2})}z_{ij}\prod_{(l,k)=(u_{1},u_{2})\uplus2}^{(m,n)}x_{lk}x_{i,m-(i-1)}\prod_{b_{1}=B_{2}+2}^{t_{2}}y_{1b_{1}}...\prod_{b_{s_{2}}=1}^{B_{s_{2}}}y_{s_{2}b_{s_{2}}}|\\
 &  & (u_{1},u_{2})<(m-1,m-1),1\leq i\leq m,0\leq B_{s_{2}}<...<B_{2}<t_{2}\},\end{eqnarray*}
\begin{eqnarray*}
 &  & \{\prod_{(i,j)=(1,1)}^{(u_{1},u_{2})}z_{ij}\prod_{(l,k)=(u_{1},u_{2})\uplus2}^{(m,n)}x_{lk}x_{i,m-(i-1)}\prod_{b_{1}=B_{2}+2}^{t_{2}}y_{1b_{1}}...\prod_{b_{s_{2}}=1}^{B_{s_{2}}}y_{s_{2}b_{s_{2}}}|\\
 &  & (m-1,m-1)\leq(u_{1},u_{2})<(m-2,m),1\leq i\leq m-1,0\leq B_{s_{2}}<...<B_{2}<t_{2}\},\end{eqnarray*}
 \begin{eqnarray*}
 &  & \{\prod_{(i,j)=(1,1)}^{(u_{1},u_{2})}z_{ij}\prod_{(l,k)=(u_{1},u_{2})\uplus2}^{(m,n)}x_{lk}\prod_{b_{1}=B_{2}+2}^{t_{2}}y_{1b_{1}}...\prod_{b_{s_{2}}=1}^{B_{s_{2}}}y_{s_{2}b_{s_{2}}}y_{m,1}|\\
 &  & (m-1,m-1)\leq(u_{1},u_{2})<(m-2,m),s_{2}<m,0\leq B_{s_{2}}<...<B_{2}<t_{2}\},\end{eqnarray*}
\begin{eqnarray*}
 &  & \{\prod_{(i,j)=(1,1)}^{(u_{1},u_{2})}z_{ij}\prod_{(l,k)=(u_{1},u_{2})\uplus2}^{(m,n)}x_{lk}\prod_{b_{1}=B_{2}+2}^{t_{2}}y_{1b_{1}}...\prod_{b_{s_{2}}=1}^{B_{s_{2}}}y_{s_{2}b_{s_{2}}}y_{m,2}|\\
 &  & (m-1,m-1)\leq(u_{1},u_{2})<(m-2,m),s_{2}=m,0\leq B_{s_{2}}<...<B_{2}<t_{2}\},\end{eqnarray*}
\begin{eqnarray*}
 &  & \{\prod_{(i,j)=(1,1)}^{(u_{1},u_{2})}z_{ij}\prod_{(l,k)=(u_{1},u_{2})\uplus2}^{(m,n)}x_{lk}x_{i,m-(i-1)}\prod_{b_{1}=B_{2}+2}^{t_{2}}y_{1b_{1}}...\prod_{b_{s_{2}}=1}^{B_{s_{2}}}y_{s_{2}b_{s_{2}}}|\\
 &  & (m-2,m)\leq(u_{1},u_{2})<(m-2,1),1\leq i\leq m-2,0\leq B_{s_{2}}<...<B_{2}<t_{2}\},\end{eqnarray*}
\begin{eqnarray*}
 &  & \{\prod_{(i,j)=(1,1)}^{(m-2,m)}z_{ij}\prod_{(l,k)=(m-1,2)}^{(m,n)}x_{lk}x_{m-1,1}\prod_{b_{1}=B_{2}+2}^{t_{2}}y_{1b_{1}}...\prod_{b_{s_{2}}=1}^{B_{s_{2}}}y_{s_{2}b_{s_{2}}}|\\
 &  & 0\leq B_{s_{2}}<...<B_{2}<t_{2}\},\end{eqnarray*}
\begin{eqnarray*}
 &  & \{\prod_{(i,j)=(1,1)}^{(m-2,m)}z_{ij}\prod_{(l,k)=(m-1,2)}^{(m,n)}x_{lk}\prod_{b_{1}=B_{2}+2}^{t_{2}}y_{1b_{1}}...\prod_{b_{s_{2}}=1}^{B_{s_{2}}}y_{s_{2}b_{s_{2}}}y_{m,2}|\\
 &  & 0\leq B_{s_{2}}<...<B_{2}<t_{2}\},\end{eqnarray*}
\begin{eqnarray*}
 &  & \{\prod_{(i,j)=(1,1)}^{(u_{1},u_{2})}z_{ij}\prod_{(l,k)=(u_{1},u_{2})\uplus2}^{(m,n)}x_{lk}\prod_{b_{1}=B_{2}+2}^{t_{2}}y_{1b_{1}}...\prod_{b_{s_{2}}=1}^{B_{s_{2}}}y_{s_{2}b_{s_{2}}}(y_{m,1},\: y_{m-1,2})|\\
 &  & (m-2,m)<(u_{1},u_{2})<(m-2,1),0\leq B_{s_{2}}<...<B_{2}<t_{2},\: s_{2}<m\},\end{eqnarray*}
\begin{eqnarray*}
 &  & \{\prod_{(i,j)=(1,1)}^{(u_{1},u_{2})}z_{ij}\prod_{(l,k)=(u_{1},u_{2})\uplus2}^{(m,n)}x_{lk}\prod_{b_{1}=B_{2}+2}^{t_{2}}y_{1b_{1}}\prod_{b_{2}=B_{3}+2}^{B_{2}}y_{2b_{2}}...\prod_{b_{s_{2}}=1}^{B_{s_{2}}}y_{s_{2}b_{s_{2}}}(y_{m,2},\: y_{m-1,1})|\\
 &  & (m-2,m)<(u_{1},u_{2})<(m-2,1),\:0\leq B_{s_{2}}<...<B_{2}<t_{2},\: s_{2}=m\mbox{ and }\\
 &  & \prod_{b_{1}=B_{2}+2}^{t_{2}}y_{1b_{1}}\prod_{b_{2}=B_{3}+2}^{B_{2}}y_{2b_{2}}...\prod_{b_{s_{2}}=1}^{B_{s_{2}}}y_{s_{2}b_{s_{2}}}\:\mbox{is divisible by }y_{m,1}\mbox{ or }y_{m-1,2}\},\end{eqnarray*}
\begin{eqnarray*}
 &  & \{\prod_{(i,j)=(1,1)}^{(u_{1},u_{2})}z_{ij}\prod_{(l,k)=(u_{1},u_{2})\uplus2}^{(m,n)}x_{lk}x_{i,m-(i-1)}\prod_{b_{1}=B_{2}+2}^{t_{2}}y_{1b_{1}}\prod_{b_{2}=B_{3}+2}^{B_{2}}y_{2b_{2}}...\prod_{b_{s_{2}}=1}^{B_{s_{2}}}y_{s_{2}b_{s_{2}}}|\\
 &  & (m-2,1)\leq(u_{1},u_{2})<(m-3,2),1\leq i\leq m-3,\:0\leq B_{s_{2}}<...<B_{2}<t_{2}\},\end{eqnarray*}
\begin{eqnarray*}
 &  & \{\prod_{(i,j)=(1,1)}^{(m-2,1)}z_{ij}\prod_{(l,k)=(m-2,3)}^{(m,n)}x_{lk}x_{m-2,2}\prod_{b_{1}=B_{2}+2}^{t_{2}}y_{1b_{1}}...\prod_{b_{s_{2}}=1}^{B_{s_{2}}}y_{s_{2}b_{s_{2}}}|\\
 &  & 0\leq B_{s_{2}}<...<B_{2}<t_{2}\},\end{eqnarray*}
\begin{eqnarray*}
 &  & \{\prod_{(i,j)=(1,1)}^{(m-2,1)}z_{ij}\prod_{(l,k)=(m-2,3)}^{(m,n)}x_{lk}\prod_{b_{1}=B_{2}+2}^{t_{2}}y_{1b_{1}}...\prod_{b_{s_{2}}=1}^{B_{s_{2}}}y_{s_{2}b_{s_{2}}}(y_{m-1,3},y_{m,1})|\\
 &  & 0\leq B_{s_{2}}<...<B_{2}<t_{2},\: s_{2}<m\},\end{eqnarray*}
\begin{eqnarray*}
 &  & \{\prod_{(i,j)=(1,1)}^{(m-2,1)}z_{ij}\prod_{(l,k)=(m-2,3)}^{(m,n)}x_{lk}\prod_{b_{1}=B_{2}+2}^{t_{2}}y_{1b_{1}}...\prod_{b_{s_{2}}=1}^{B_{s_{2}}}y_{s_{2}b_{s_{2}}}(y_{m-1,3},y_{m,2})|\\
 &  & 0\leq B_{s_{2}}<...<B_{2}<t_{2}\mbox{ and }s_{2}=m,\prod_{b_{1}=B_{2}+2}^{t_{2}}y_{1b_{1}}\prod_{b_{2}=B_{3}+2}^{B_{2}}y_{2b_{2}}...\prod_{b_{s_{2}}=1}^{B_{s_{2}}}y_{s_{2}b_{s_{2}}}\\
 &  & \mbox{is divisible by }y_{m,1}\},\end{eqnarray*}
\begin{eqnarray*}
 &  & \{\prod_{(i,j)=(1,1)}^{(m-3,m)}z_{ij}\prod_{(l,k)=(m-2,2)}^{(m,n)}x_{lk}x_{m-2,1}\prod_{b_{1}=B_{2}+2}^{t_{2}}y_{1b_{1}}...\prod_{b_{s_{2}}=1}^{B_{s_{2}}}y_{s_{2}b_{s_{2}}}|\\
 &  & 0\leq B_{s_{2}}<...<B_{2}<t_{2}\},\end{eqnarray*}
\begin{eqnarray*}
 &  & \{\prod_{(i,j)=(1,1)}^{(m-3,m)}z_{ij}\prod_{(l,k)=(m-2,2)}^{(m,n)}x_{lk}\prod_{b_{1}=B_{2}+2}^{t_{2}}y_{1b_{1}}...\prod_{b_{s_{2}}=1}^{B_{s_{2}}}y_{s_{2}b_{s_{2}}}(y_{m-1,3,}y_{m,2})|\\
 &  & 0\leq B_{s_{2}}<...<B_{2}<t_{2}\},\end{eqnarray*}
\begin{eqnarray*}
 &  & \{\prod_{(i,j)=(1,1)}^{(u_{1},u_{2})}z_{ij}\prod_{(l,k)=(u_{1},u_{2})\uplus2}^{(m,n)}x_{lk}\prod_{b_{1}=B_{2}+2}^{t_{2}}y_{1b_{1}}...\prod_{b_{s_{2}}=1}^{B_{s_{2}}}y_{s_{2}b_{s_{2}}}(y_{m,1},y_{m-1,2},y_{m-2,3})|\\
 &  & (m-3,m-1)\leq(u_{1},u_{2})<(m-3,2),0\leq B_{s_{2}}<...<B_{2}<t_{2}\},\\
 &  & ...\\
 &  & \{\prod_{(l,k)=(1,2)}^{(m,n)}x_{lk}\prod_{b_{1}=B_{2}+2}^{t_{2}}y_{1b_{1}}\prod_{b_{2}=B_{3}+2}^{B_{2}}y_{2b_{2}}...\prod_{b_{s_{2}}=1}^{B_{s_{2}}}y_{s_{2}b_{s_{2}}}(y_{2,m},y_{3,m-1},...,y_{m,2})|\\
 &  & 0\leq B_{s_{2}}<...<B_{2}<t_{2}\}).\end{eqnarray*}

Notice that all the elements are in the same degree $mm-1+1+t_{2}-(s_{2}-1)$.
We find a filtration starting from $(\mbox{in}(\mathcal{L}))^{*}$
and ending at $(x_{m,1},x_{m-1,2},....,x_{1m})$. Each quotient of
this filtration will have the form $P/PL$, where $P$ is an ideal
generated in the same degree and $L$ is an ideal generated by variables
such that those variables form a regular sequence modulo $P$. 

We look at the variable $y_{s_{2},1}$. The elements in $(\mbox{in}(\mathcal{L}))^{*}$
that are divisible by $y_{s_{2},1}$ must have $B_{s_{2}}>0$. The
elements in $\mbox{(in}(\mathcal{L}))^{*}$ that are not divisible
by $y_{s_{2},1}$ must have $B_{s_{2}}=0$. Hence $(\mbox{in}(\mathcal{L})^{*})=(A_{1}(y_{s_{2},1},...,y_{s_{2}-1,2,}..y_{1,s_{2})}),C_{1}y_{s_{2},1})$,
where the elements in $C_{1}$ have $B_{s_{2}}>1$. Also all the elements
of $A_{1}$ and $C_{1}$ are in the same degree. Furthermore, $A_{1}\cap C_{1}=A_{1}c_{1}=C_{1}(\{a_{1i}\})$,
$(y_{s_{2},1},...,y_{1,s_{2}})$ is a regular sequence modulo $A_{1}$
and $(y_{s_{2},1},\{a_{1i}\})$ is a regular sequence modulo $C_{1}$.
We look at the following filtration\[
(\mbox{in}(\mathcal{L}))^{*}\subset(A_{1},C_{1}y_{s_{2},1})\subset(A_{1},C_{1}).\]
Then we can use the proof of Lemma \ref{REGX} to show that the quotients
are $A_{1}/A_{1}(y_{s_{2},1},...,y_{1,s_{2}})$ and $C_{1}/C_{1}(\{a_{1i}\},\: y_{s_{2},1})$.
Notice the following equalities: $\mbox{reg}(A_{1}/A_{1}(y_{s_{2},1},...,y_{1,s_{2}}))=\mbox{reg}(A_{1})$
and $\mbox{reg}(C_{1}/C_{1}(\{a_{1i}\},y_{s_{2},1})=\mbox{reg}(C_{1})$. 

Next, we look at $y_{s_{2},2}$ and write $(A_{1},C_{1})=(A_{2}(y_{s_{2},2},....,y_{1,s_{2}+1}),C_{2}y_{s_{2},2})$.
As before we have a filtration \[
(A_{1},C_{1})\subset(A_{2},C_{2}y_{s_{2},2})\subset(A_{2},C_{2}),\]
and the quotients are $A_{2}/A_{2}(y_{s_{2},2},...,y_{1,s_{2}+1})$
and $C_{2}/C_{2}(\{a_{2i}\},y_{s_{2},2})$. We can continue to $y_{s_{2},3}$
until $y_{s_{2},t_{2}-(s_{2}-1)}$, we will reduce to an ideal $J_{1}$
which is generated in the same degree $mm-1+1$. 

Now we look at the variable $z_{1,1}$. When $(u_{1},u_{2})=(0,0)$
in an element, $z_{1,1}$ is not a factor of this element. When $(u_{1},u_{2})\leq(1,1)$,
then $z_{1,1}$ is a factor. The ideal $J_{1}$ can be written as
$(D_{1,1}(z_{1,1},\{d_{1,1,i}\}),\, E_{1,1}z_{1,1})$. Hence we have
a filtration\[
J_{1}\subset(D_{1,1},E_{1,1}z_{1,1})\subset(D_{1,1},E_{1,1})\]
 with quotients $D_{1,1}/D_{1,1}(z_{1,1},\{d_{1,1,i}\})$ and $E_{1,1}/E_{1,1}(\{e_{1,1,i}\},z_{1,1})$.
We look at $z_{1,2}$ next and reduce to an ideal $(D_{1,2},E_{1,2})$,
we continue to $z_{m,m-1}$. We can find a filtration from $J_{1}$
to $(D_{m,m-1},E_{m,m-1})=(x_{m,1},....,x_{1m})$. 

Since $\mbox{reg}(x_{m,1},....,x_{1,m})=\mbox{reg}(D_{m,m-1},E_{m,m-1})=1$,
it follows that $\mbox{reg}(D_{m,m-1})$ and $\mbox{reg}(E_{m,m-1})$
are equal to 1. Also the final quotient \[
(D_{m,m-1},E_{m,m-1})/(D_{m,m-1},E_{m,m-1}z_{m,m-1})\]
 has regularity equals to the regularity of $E_{m,m-1}$. It follows
that \[
\mbox{reg}(D_{m,m-1},E_{m,m-1}z_{m,m-1})=2.\]
 Since the quotient $(D_{m,m-1},E_{m,m-1}z_{m,m-1})/(D_{m,m-2},E_{m,m-2})$
has regularity equal to $\mbox{reg}(D_{m,m-1})=1$, we have $\mbox{reg}(D_{m,m-2},E_{m,m-2})=2$.
Continue with the same argument, we obtain $\mbox{reg}(J_{1})=1+mm-1$.
Hence $\mbox{reg}((\mbox{in}\mathcal{L})^{*})=1+mm-1+t_{2}-(s_{2}-1)$. 

For the induction steps, we write $I_{n}$ for the Alexander dual
ideal $(\mbox{in}(\mathcal{L}))^{*}$ in the $m$ by $n$ matrix case.
We assume by induction hypothesis that $\mbox{reg}(I_{n-1})=\mbox{d}(I_{n-1})=m(n-1)-1+t_{1}-(s_{1}-1)+t_{2}-(s_{2}-1)$.
We will show that we can build a filtration from $I_{n}$ to $I_{n-1}$
such that the quotients are the form $P/PL$ where $P$ is an ideal
and $L$ is an ideal generated by variables such that they are a regular
sequence modulo $P$. From Lemma \ref{ADL} we know that each summand
of $\mbox{in}(\mathcal{L})$ is generated by monomials satisfying
the assumption of Lemma \ref{ADX}. For a fixed variable $y_{ij}$,
we can write $(h_{Y})^{*}=(A^{Y}y_{ij},A^{Y}(\{a_{l}^{Y}\}),B^{Y}y_{ij})$.
Also $A^{Y}\cap B^{Y}=B^{Y}(\{b_{i}^{Y}\})$, where $\{b_{i}^{Y}\}$
are variables such that $(y_{ij},\{b_{i}^{Y}\})$ is a regular sequence
modulo $B^{Y}$ and $(y_{ij},\{y_{lk}\})$ is a regular sequence modulo
$A^{Y}$. Similarly, $(h_{U})^{*}=(A^{U}y_{ij},A^{U}(\{y_{lk},x_{p,q},z_{uv}\}),B^{U}y_{ij})$,
$(h_{W})^{*}=(A^{W}y_{ij},A^{W}(\{y_{lk},x_{p,q},z_{u,v}\}),B^{W}y_{ij})$,
and all other components of $(\mbox{in}(\mathcal{L}))^{*}$ that involve
$y_{ij}$. For $(h_{X})^{*}$ that does not involve $y_{ij}$ we leave
as it is and similarly for others that do not involve $y_{ij}$. 

\begin{flushleft}
Claim: There is a filtration from $\mbox{in}(\mathcal{L})^{*}$ to
an ideal \[
I_{y_{ij}}=(h_{X})^{*}\cap(A^{Y},B^{Y})\cap(h_{g})^{*}\cap(A^{U},B^{U})\cap...\cap(A^{W},B^{W})\]
 such that the quotients are the form $P/PL$ where $P$ is an ideal
and $L$ is an ideal generated by variables such that they are a regular
sequence modulo $P$.
\par\end{flushleft}

With this claim, we can continue picking another variable and reduce
to a bigger ideal that does not involve the new variable. We can continue
the process until we reach an ideal that does not involve any $z_{i,n}$,
or $x_{i,n}$ or $y_{i,n}$. This ideal is $I_{n-1}$. 

\begin{flushleft}
Proof of claim: Without lost of generality, we just need to show:
there is a filtration from $(A^{Y}(y_{ij},\{a_{l}^{Y}\}),B^{Y}y_{ij})\cap(A^{U}(y_{ij},\{a_{l}^{U}\}),B^{U}y_{ij}))$
to $(A^{Y},B^{Y})\cap(A^{U},B^{U})$ with quotients are the form $P/PL$
as above. For convenience, we write $A^{Y}(\{a_{l}^{Y}\})=A^{Y}C$
and $A^{U}(\{a_{l}^{U}\})=A^{U}D$. Then we look at the following
filtration: \begin{eqnarray*}
 & (A^{Y}y_{ij},A^{Y}C,B^{Y}y_{ij})\cap(A^{U}y_{ij},A^{U}D,B^{U}y_{ij})\\
= & \qquad\qquad((A^{Y}\cap A^{U})y_{ij},(A^{Y}\cap A^{U})Cy_{ij},(B^{Y}\cap A^{U})y_{ij},(A^{Y}\cap A^{U})CD,\quad\\
 & (A^{Y}\cap B^{U})y_{ij},(B^{Y}\cap B^{U})y_{ij})\qquad\qquad & =:J_{0}\end{eqnarray*}
\begin{eqnarray*}
\subset & ((A^{Y}\cap A^{U})y_{ij},(A^{Y}\cap A^{U})C,(B^{Y}\cap A^{U})y_{ij},(A^{Y}\cap B^{U})y_{ij},(B^{Y}\cap B^{U})y_{ij}) & =:J_{1}\\
\subset & ((A^{Y}\cap A^{U}),(B^{Y}\cap A^{U})y_{ij},(A^{Y}\cap B^{U})y_{ij},(B^{Y}\cap B^{U})y_{ij}) & =:J_{2}\\
\subset & ((A^{Y}\cap A^{U}),(B^{Y}\cap A^{U}),(A^{Y}\cap B^{U})y_{ij},(B^{Y}\cap B^{U})y_{ij}) & =:J_{3}\\
\subset & ((A^{Y}\cap A^{U}),(B^{Y}\cap A^{U}),(A^{Y}\cap B^{U}),(B^{Y}\cap B^{U})y_{ij}) & =:J_{4}\\
\subset & ((A^{Y}\cap A^{U}),(B^{Y}\cap A^{U}),(A^{Y}\cap B^{U}),(B^{Y}\cap B^{U}))=(A^{Y},B^{Y})\cap(A^{U},B^{U}) & =:J_{5}.\end{eqnarray*}

\par\end{flushleft}

\begin{flushleft}
We have \begin{eqnarray*}
 & J_{1}/J_{0}\\
= & (A^{Y}\cap A^{U})C/((A^{Y}\cap A^{U})CD,(A^{Y}\cap A^{U})C\cap(A^{Y}\cap A^{U})y_{ij},(A^{Y}\cap A^{U})C\cap\\
 & (B^{Y}\cap A^{U})y_{ij},(A^{Y}\cap A^{U})C\cap(A^{Y}\cap B^{U})y_{ij},(A^{Y}\cap A^{U})C\cap(B^{Y}\cap B^{U})y_{ij})\\
= & (A^{Y}\cap A^{U})C/((A^{Y}\cap A^{U})C(D,y_{ij}),(A^{Y}\cap B^{Y})\cap A^{U}Cy_{ij},(A^{Y}\cap A^{U}\cap B^{U})Cy_{ij},\\
 & (A^{Y}\cap B^{Y})\cap(A^{U}\cap B^{U})Cy_{ij})\\
= & (A^{Y}\cap A^{U})C/((A^{Y}\cap A^{U})C(D,y_{ij}),(A^{Y}a^{Y})\cap A^{U}Cy_{ij},(A^{Y}\cap A^{U}a^{U})Cy_{ij},\\
 & (A^{Y}a^{Y})\cap(A^{U}a^{U})Cy_{ij})\\
= & (A^{Y}\cap A^{U})C/((A^{Y}\cap A^{U})C(D,y_{ij})).\end{eqnarray*}
Hence $\mbox{reg}(J_{1}/J_{0})=\mbox{reg}(A^{Y}\cap A^{U})C$. Also
\begin{eqnarray*}
 & J_{2}/J_{1}\\
= & (A^{Y}\cap A^{U})/((A^{Y}\cap A^{U})(C,y_{ij}),(A^{Y}\cap A^{U})\cap(B^{Y}\cap A^{U})y_{ij},(A^{Y}\cap A^{U})\cap\\
 & (A^{Y}\cap B^{U})y_{ij},(A^{Y}\cap A^{U})\cap(B^{Y}\cap B^{U})y_{ij})\\
= & (A^{Y}\cap A^{U})/((A^{Y}\cap A^{U})(C,y_{ij}),(A^{Y}\cap B^{Y}\cap A^{U})y_{ij},(A^{Y}\cap A^{U}\cap B^{U})y_{ij},\\
 & (A^{Y}\cap B^{Y})\cap(A^{U}\cap B^{U})y_{ij})\end{eqnarray*}
\begin{eqnarray*}
= & (A^{Y}\cap A^{U})/((A^{Y}\cap A^{U})(C,y_{ij}),(A^{Y}a^{Y}\cap A^{U})y_{ij},(A^{Y}\cap A^{U}a^{U})y_{ij},\\
 & (A^{Y}a^{Y})\cap(A^{U}a^{U})y_{ij})\\
= & (A^{Y}\cap A^{U})/((A^{Y}\cap A^{U})(C,y_{ij})).\end{eqnarray*}
Hence $\mbox{reg}(J_{2}/J_{1})=\mbox{reg}(A^{Y}\cap A^{U})$. Similarly
for $ $\begin{eqnarray*}
 & J_{3}/J_{2}\\
= & (B^{Y}\cap A^{U})/((B^{Y}\cap A^{U})y_{ij},(B^{Y}\cap A^{U})\cap(A^{Y}\cap A^{U}),(B^{Y}\cap A^{U})\cap(A^{Y}\cap B^{U})y_{ij},\\
 & (B^{Y}\cap A^{U})\cap(B^{Y}\cap B^{U})y_{ij})\\
= & (B^{Y}\cap A^{U})/((B^{Y}\cap A^{U})y_{ij},(B^{Y}\cap A^{Y}\cap A^{U}),(B^{Y}\cap A^{Y})\cap(A^{U}\cap B^{U})y_{ij},\\
 & (B^{Y}\cap A^{U}\cap B^{U})y_{ij})\\
= & (B^{Y}\cap A^{U})/((B^{Y}\cap A^{U})y_{ij},(B^{Y}(\{b_{i}^{Y}\})\cap A^{U}),(B^{Y}(\{b_{i}^{Y}\}))\cap(A^{U}a^{U})y_{ij},\\
 & (B^{Y}\cap A^{U}a^{U})y_{ij})\\
= & (B^{Y}\cap A^{U})/((B^{Y}\cap A^{U})(\{b_{i}^{Y}\},y_{ij})).\end{eqnarray*}
Thus $\mbox{reg}(J_{3}/J_{2})=\mbox{reg}(B^{Y}\cap A^{U})$. Similarly
for \begin{eqnarray*}
 & J_{4}/J_{3}\\
= & (A^{Y}\cap B^{U})/((A^{Y}\cap B^{U})y_{ij},(A^{Y}\cap B^{U})\cap(A^{Y}\cap A^{U}),(A^{Y}\cap B^{U})\cap(B^{Y}\cap A^{U}),\\
 & (A^{Y}\cap B^{U})\cap(B^{Y}\cap B^{U})y_{ij})\\
= & (A^{Y}\cap B^{U})/((A^{Y}\cap B^{U})y_{ij},(A^{Y}\cap B^{U}(\{b_{i}^{Y}\})),(A^{Y}a^{Y}\cap B^{U}(\{b_{i}^{Y}\})),\\
 & (A^{Y}a^{Y}\cap B^{U}y_{ij}))\\
= & (A^{Y}\cap B^{U})/((A^{Y}\cap B^{U})(y_{ij},\{b_{i}^{Y}\})).\end{eqnarray*}
Therefore $\mbox{reg}(J_{4}/J_{3})=\mbox{reg}(A^{Y}\cap B^{U})$.
Finally, \begin{eqnarray*}
 & J_{5}/J_{4}\\
= & (B^{Y}\cap B^{U})/((B^{Y}\cap B^{U})y_{ij},(B^{Y}\cap B^{U})\cap(A^{Y}\cap A^{U}),(B^{Y}\cap B^{U})\cap\\
 & (B^{Y}\cap A^{U}),(B^{Y}\cap B^{U})\cap(A^{Y}\cap B^{U}))\\
= & (B^{Y}\cap B^{U})/((B^{Y}\cap B^{U})y_{ij},(B^{Y}(\{b_{i}^{Y}\})\cap B^{U}(\{b_{i}^{U}\})),\\
 & (B^{Y}\cap B^{U}(\{b_{i}^{U}\})),(B^{Y}(\{b_{i}^{Y}\})\cap B^{U}))\\
= & (B^{Y}\cap B^{U})/((B^{Y}\cap B^{U})(y_{ij},\{b_{i}^{Y}\},\{b_{i}^{U}\}).\end{eqnarray*}
We obtain $\mbox{reg}(J_{5}/J_{4})=\mbox{reg}(B^{Y}\cap B^{U}).$ 
\par\end{flushleft}

Since by induction hypothesis we have \[
\mbox{reg}(A^{Y},B^{Y})\cap(A^{U},B^{U})=d=\mbox{d}(A^{Y},B^{Y})\cap(A^{U},B^{U}),\]
 it follows that \[
\mbox{d}(A^{Y}\cap A^{U})=\mbox{d}(B^{Y}\cap A^{U})=\mbox{d}(A^{Y}\cap B^{U})=\mbox{d}(B^{Y}\cap B^{U})=d\]
and \[
\mbox{reg}(A^{Y}\cap A^{U})=\mbox{reg}(B^{Y}\cap A^{U})=\mbox{reg}(A^{Y}\cap B^{U})=\mbox{reg}(B^{Y}\cap B^{U})=d.\]

Also notice that $\mbox{reg}(A^{Y}\cap A^{U})C\geq\mbox{d}(A^{Y}\cap A^{U})C\geq d+1.$
We use the regularity of the quotients of the filtration. to obtain
the regularity of $J_{0}$. We have $\mbox{reg}J_{4}=\mbox{reg}(J_{5})+1=d+1$
and $\mbox{reg}J_{3}=d+1=\mbox{reg}J_{2}=d+1=\mbox{reg}J_{1}=d+1$.
Notice $\mbox{reg}J_{0}\geq\mbox{deg}J_{0}\geq\mbox{deg}J_{1}+1=d+1$.
We will show $\mbox{reg}J_{0}\leq d+1$, hence $\mbox{reg}J_{0}=\mbox{deg}J_{0}=d+1$.
Assume $\mbox{reg}J_{0}>d+2$, then $\mbox{reg}J_{1}=\mbox{max}\{\mbox{reg}J_{0},\mbox{reg}J_{1}/J_{0}=\mbox{reg}(A^{Y}\cap A^{U})C\}>d+2$,
a contradiction. Hence $\mbox{reg}J_{0}=d+1=\mbox{deg}J_{0}$. This
completes the proof of the claim. 

Since $s_{1}\geq s_{2}$, we assume $s_{1}=m$ and we observe that
$(\mbox{in}(\mathcal{L}))^{*}$ does not involve $y_{in}$ for $i>2$
and $z_{mn}$. If $t_{2}<n$, then $(\mbox{in}(\mathcal{L}))^{*}$
does not involve $y_{in}$ for all $i$. By using the claim above,
we can find a filtration starting from $(\mbox{in}(\mathcal{L}))^{*}$
to an ideal $J_{y_{1n}}$, where $y_{1n}$ is not a factor of the
minimal monomial generators of $J_{y_{1n}}$. Then we continue the
filtration to an ideal $J_{2n}$ such that $y_{2n}$ is not a factor
of the minimal monomial generators of $J_{y_{2n}}$. We need those
two steps when $t_{2}=n$, otherwise we skip those steps. The next
step is to look at $z_{m-1,n}$ and find a filtration until an ideal
$J_{z_{,m-1},n}$ such that $z_{m-1,n}$ is not a factor of $J_{z_{m-1,n}}$.
Next we consider $z_{m-2,n}$ and continue to $z_{1n}$. Finally,
we consider $x_{1n}$. Then we will get the ideal $I_{n-1}$. We will
have a filtration as follow:\[
I_{n}\subset J_{y_{1n}}\subset J_{y_{2n}}\subset J_{m-1,n}\subset J_{m-2,n}\subset....\subset J_{x_{1n}}=I_{n-1}.\]

Hence by using a similar argument as in the proof of the claim, we
have 

\begin{flushleft}
\begin{eqnarray*}
\mbox{reg}I_{n} & = & \mbox{reg}I_{n-1}+2+m-1+1\\
 & = & m(n-1)-1+t_{1}-1-(s_{1}-1)+t_{2}-1-(s_{2}-1)+m+2\\
 & = & mn-1+t_{1}-(s_{1}-1)+t_{2}-(s_{2}-1)\\
 & = & \mbox{d}I_{n-1}+2+m-1+1\\
 & = & \mbox{d}I_{n}.\end{eqnarray*}
This complete the proof of this Lemma.
\par\end{flushleft}
\end{proof}
\medskip{}

We are now ready to prove Theorem \ref{CM}.

\medskip{}

\begin{flushleft}
\textit{Proof of Theorem }\ref{CM}: We know that $k[X,Y,Z]/(\mathrm{in}(\mathcal{K})$
is Cohen-Macaulay by Lemma \ref{REGL} and Theorem \ref{ER}. Hence
$\mathcal{R}(\mathbb{D})=k[X,Y,Z]/\mathcal{K}$ is Cohen-Macaulay
{[}E{]}.\hfill{}$\square$
\par\end{flushleft}

\address{\begin{center}
{\small Department of Mathematics, }
\par\end{center}}

\address{\begin{center}
{\small University of California, Riverside, CA 92521, USA}
\par\end{center}}

\email{\begin{center}
e-mail: linkuei@ucr.edu
\par\end{center}}
\end{document}